\documentclass[12pt]{amsart}
\usepackage[latin1]{inputenc}
\usepackage{amsmath,amsthm,amssymb}
\usepackage{mathrsfs}
\usepackage{color}
\usepackage[colorlinks]{hyperref}

\usepackage{enumerate}

\usepackage{algorithmic}
\usepackage[ruled]{algorithm}
\usepackage{caption}

\newtheorem{theorem}{Theorem}[section]
\newtheorem{lemma}[theorem]{Lemma}
\newtheorem{proposition}[theorem]{Proposition}

\newtheorem{corollary}[theorem]{Corollary}
\theoremstyle{definition}

\theoremstyle{remark}
\newtheorem{remark}[theorem]{Remark}

\newtheorem{notations}[theorem]{Notations}

\newcommand{\oc}{\ensuremath{\mathcal{O}}}

\newcommand{\fc}{\ensuremath{\mathcal{F}}}
\newcommand{\dc}{\ensuremath{\mathcal{D}}}

\newcommand{\hc}{\ensuremath{\mathcal{H}}}

\newcommand{\Sc}{\ensuremath{\mathcal{S}}}

\newcommand{\kc}{\ensuremath{\mathcal{K}}}
\newcommand{\nc}{\ensuremath{\mathcal{N}}}

\newcommand{\pc}{\ensuremath{\mathcal{P}}}

\newcommand{\mc}{\ensuremath{\mathcal{M}}}
\newcommand{\qc}{\ensuremath{\mathcal{Q}}}

\newcommand{\bH}{\mathbb{H}}
\newcommand{\bZ}{\mathbb{Z}}

\newcommand{\bF}{\mathbb{F}}

\newcommand{\bC}{\mathbb{C}}
\newcommand{\bCl}{\mathbb{C}^l}
\newcommand{\bQ}{\mathbb{Q}}

\newcommand{\bG}{\mathbb{G}}

\def\bin #1#2 {\left( \matrix { #1 \cr #2 \cr } \right) }

\newcommand{\tS}{\widetilde{{\ensuremath{\mathcal{S}}}}}
\newcommand{\tD}{\widetilde{\Delta}}
\newcommand{\tDpq}{\widetilde{\Delta}_{pq}}

\begin{document}

\title[Explicit Decomposition Theorem for special Schubert varieties.]
{Explicit Decomposition Theorem for special Schubert varieties.}

%    Information for  author
\author{Davide Franco }
\address{Universit\`a di Napoli
\lq\lq Federico II\rq\rq, 
Dipartimento di Matematica e
Applicazioni \lq\lq R. Caccioppoli\rq\rq,  Via
Cintia, 80126 Napoli, Italy} \email{davide.franco@unina.it}

\abstract We provide an explicit canonical description of the perverse cohomology sheaves and of the primitive perverse cohomology complexes
for the non-small
resolution $\pi: \tilde{\mathcal S} \to \mathcal S$ of a Special Schubert variety $\mathcal S$. For such a resolution, we also discuss a way to obtain an 
explicit splitting 
of $R\pi{_*}\bQ_{\tS}$, in the derived category, by means of Gysin morphisms and 
cohomology extensions.



\endabstract
\maketitle

\bigskip
\section{Introduction}

The Decomposition Theorem is a beautiful and very deep result about
algebraic maps. In the words of MacPherson \lq\lq it contains as
special cases the deepest homological properties of algebraic maps
that we know\rq\rq  \cite{Mac83}, \cite{Williamson}. 
In the literature
one can find different approaches to the Decomposition Theorem
\cite{BBD}, \cite{DeCMHodge}, \cite{DeCMReview}, \cite{Saito},
{\cite{Williamson}, which is a very general result but also rather implicit.

However,  there are many special cases for which the Decomposition
Theorem admits a simplified and explicit approach. One of these is  the case of
varieties with isolated singularities \cite{N}, \cite{DGF3}. For instance,
in our previous work with V. Di Gennaro \cite{DGF3}, we developed a simplified approach to the Decomposition Theorem for varieties 
with isolated singularities, in connection with the existence of a \textit{natural Gysin morphism}, as defined in \cite[Definition 2.3]{DGF1}. 

Another  case for which the Decomposition Theorem admits an explicit statement was developed in \cite{DGF2strata}, where we
 assumed 
 $X$ to be a complex, irreducible, quasi-projective variety {\it  with two strata}, that is to say
equipped  with a resolution consisting in a single blow-up of a smooth locus.
Our main result was a  short and explicit proof of the Decomposition Theorem \cite[Theorem 3.1]{DGF2strata}, and an explicit way to compute the intersection cohomology of 
$X$ by means  of the cohomology of  the exceptional locus. We also applied our formalism to    {\it Special Schubert varieties with two strata} \cite[\S 4]{DGF2strata}  and to 
general hypersurfaces containing a smooth threefold   $T\subset \mathbb P^5$.

What we are going to do in this paper is to extend a similar explicit approach  to {\it Special Schubert varieties with an arbitrary number of strata}.

Let $\mathcal S$ be a {\it single condition Schubert variety or special Schubert variety} of dimension $n$
\cite[p. 328]{CGM}, \cite[Example 8.4.9]{Kirwan}. As  is well known, $\Sc$ admits two standard resolutions: a
{\it small resolution} $\xi: \dc \to \Sc$ \cite[Definition 8.4.6]{Kirwan} and a (usually) non-small one
$\pi: \tS \to \Sc$ \cite[\S 3.4 and Exercise 3.4.10]{Manivel}. We will describe both resolutions $\pi $ and $\xi $ in \S 2 and \S 3, respectively.
By \cite[\S 6.2]{GMP2} and \cite[Theorem 8.4.7]{Kirwan}, we have 
\begin{equation}\label{decsmall}
IC^{\bullet}_{\Sc}\simeq R\xi {_*}\bQ_{\dc}[n] \quad\text{{\rm  in}} \quad D^b_c(\Sc),
\end{equation}
where $IC^{\bullet}_{\Sc}$
denotes the \textit{intersection cohomology complex} of $\Sc$ \cite[p. 156]{Dimca2},  and $D^b_c(\Sc)$ is the constructible derived category  
 of sheaves of $\mathbb Q$-vector spaces on $\Sc$. 

One of the main consequences of the Decomposition Theorem
is  that the intersection cohomology complex of $\Sc$ is also a direct summand of 
$R\pi {_*}\bQ_{\tS}[n]$ in $D^b_c(\Sc)$. Specifically, the Decomposition Theorem says that there is a decomposition in $D^b_c(\Sc)$ \cite[Theorem 1.6.1] {DeCMReview} 
\begin{equation}\label{decomposition}
R\pi {_*}\bQ_{\tS}[n]\cong \oplus_{i\in \bZ }\sideset{_{}^{p}}{_{}^{i}}{\mathop \hc} (R\pi {_*}\bQ_{\tS}[n])[-i],
\end{equation}
where
$\sideset{_{}^{p}}{_{}^{i}}{\mathop \hc}(R\pi {_*}\bQ_{\tS}[n])$ denote the {\it perverse cohomology sheaves}
\cite[\S 1.5]{DeCMReview}. Furthermore, the perverse sheaves $\sideset{_{}^{p}}{_{}^{i}}{\mathop \hc}(R\pi {_*}\bQ_{\tS}[n])$ are semisimple, i.e. direct sum of  intersection cohomology complexes
of semisimple local systems, supported in the smooth strata of $\Sc$.  
Our variety $\mathcal S$ is equipped with  a  standard Whitney stratification: 
\begin{equation}\label{flagintr}
\Delta_1\subset\Delta_2\subset\dots\subset\Delta_r\subset\Delta_{r+1}=\mathcal S,
\end{equation}
where
 each $\Delta_p$ is a special Schubert variety as well and such that
\begin{equation}\label{desingstrata}
\pi^{-1}(\Delta_p^0)\to
\Delta_p^0, \quad \Delta_p^0:=\Delta_p \backslash {\text{Sing}}(\Delta_p),
\end{equation}
is a smooth and proper fibration on $\Delta_p^0$, with fibre $F_p$ (\ref{Fiber}).
In other words, the stratification
(\ref{flagintr})
 is a
{\it stratification of} $\Sc$ {\it adapted to} $\pi$ \cite{CGM},
\cite{Williamson}.

\textit{The main aim of this paper is to determine the summands  involved in (\ref{decomposition})  and to provide an explicit description of  the splitting. } 

The discussion of the summands in (\ref{decomposition}), which is a somewhat easier task, is accomplished in
 section 3. The starting point of our analysis stems from the remark that \textit{the semisimple local systems involved in the decomposition are constant sheaves supported in the smooth 
 strata} $\Delta _p^0$. In other words, the decomposition (\ref{decomposition})  takes the following form

\begin{equation}\label{decimpr}
R\pi {_*}\bQ_{\tS}\cong \bigoplus_{p, q}IC_{\Delta_p}^{\bullet}[q]^{\oplus m_{pq}}
\end{equation}
for suitable multiplicities $m_{pq}\in \mathbb N_0$. As explained in \S 3, this follows just combining the Decomposition Theorem with  the global invariant cycle theorem, taking also into account of the
resolution of the strata (\ref{desingstrata}).

Owing  of  the splitting (\ref{decimpr}), 
 the only thing we need to do is the computation of the occurring multiplicities $m_{pq}$. This  will be achieved in Theorem \ref{PerverseCohomology} 
by means of a combinatoric argument based on dimension of the stalk cohomology groups
$$\hc^{\bullet}(R\pi {_*}\bQ_{\tS})_y \cong H^{\bullet}(\pi ^{-1}(y)) \quad{\text{and}} $$
$$\hc^{\bullet}(IC_{\Delta_p}^{\bullet})_y \cong H^{\bullet + \dim \Delta_p}(\xi_p ^{-1}(y)), $$
where $ \xi _p: \dc _p \to \Delta_p$ denotes the small resolution of the stratum $\Delta_p$.
On that basis, the proof of Theorem \ref{PerverseCohomology}
reduces to a careful dimensional counting.

It will turn out  that the perverse cohomology sheaf supported on a given stratum $\Delta_p$ is completely controlled 
by the cohomology of a sub-Grassmannian $T_p$ of the fiber $F_p$ (cf. Notations \ref{Dpq}). Also the relative Hard Lefschetz theorem for perverse cohomology sheaves
coincide with the  classical Hard Lefschetz theorem for $T_p$ (\ref{Hard}). 

The rest of the paper is devoted to a more subtle discussion of the splittings.
We begin in  section 5, where we also provide a
 geometric explanation of the results of \S 3. Therein we  describe the splitting (\ref{decomposition}) by means of the Gysin morphism and the 
Leray-Hirsch Theorem.
By Leray-Hirsch Theorem and Proper Base Change (\S 5 and \cite[Theorem 2.3.26]{Dimca2}), the complex $R\pi {_*}\bQ_{\tS}[n]\mid_{\Delta_p^0}$ can be explicitly described   as
a  direct sum of shifted trivial local systems on $\Delta_p^0$, with summand  in one-to-one correspondence with the cohomology classes of $F_p$:

\begin{equation}
\label{LHintro}
\bigoplus_{\alpha=0}^{2dimF_p} H^{\alpha}(F_p)\otimes \mathbb Q_{\Delta_p^0}[-\alpha]
\cong R{\pi}{_*}\mathbb  Q_{\mathcal S}\mid_{\Delta_p^0}.
\end{equation} Such a description is 
\textit{non-canonical and depends on the choice of a cohomology extension}. In Remarks \ref{cohomologyextension} and \ref{differentcohext} 
 we  discuss the effect that a change of the cohomology extension has on the splitting (\ref{LHintro}).
Furthermore, the Gysin morphism from a standard resolution of $\pi^{-1}(\Delta_p)$ allows us to define suitable maps (also depending on the cohomology extension)

$$
a\otimes_{\mathbb Q}
IC_{\Delta_p}^{\bullet}[-s - \alpha]
\rightarrow  R{\pi}_{*}\mathbb Q_{\mathcal S}[l], \quad \forall a\in H^{\alpha}(F_p),
$$
where $l:=n - \dim \pi^{-1}(\Delta_p)$ and $s:=\dim \Delta_p$.
It will turn out that the image of the previous morphisms,
with $a$ varying among the cohomology classes of the fibre $F_p$ that are not killed by the top Chern class of the normal bundle of 
$\pi^{-1}(\Delta_p^0)$ in $\tS$, generate the perverse cohomology sheaves, thus providing a \textit{canonical} (i.e. independent on the cohomology extension) isomorphism:

\begin{equation}
\label{perverseintro}
\sideset{_{}^{p}}{_{}^{i}}{\mathop \hc}(R\,\pi _{ *}
\mathbb Q_{\widetilde{\mathcal S}}[n])_{\Delta_p} \cong     
H^{\dim T_p + i}(T_p)\otimes_{\mathbb Q}
 IC_{\Delta_p}^{\bullet},
\end{equation}
(cf. Theorem  \ref{isoDE}, Remark \ref{commento} and (\ref{Pieridim})). Along the way, we also obtain an explicit splitting (\ref{decomposition}), depending on the choice of the cohomology extensions
(cf. Theorem \ref{Gysinsplitting}).

Having defined an explicit non-canonical splitting, we will turn in 
sections 6 and 7 to the task of defining a canonical one. In \S 6 we will prove that  the \textit{primitive perverse cohomology complexes} can be canonically
defined by means of the primitive cohomology of the sub-Grassmannians $T_{p}$. Furthermore, we will prove that the canonical morphism defined by Deligne in \cite[Proposition 2.4]{Deligne}
can be obtained as in \S 5 as soon as the cohomology extensions are defined appropriately (cf. Theorem \ref{Deligne}).
So, the first Deligne splitting (\cite[2.5.1]{Deligne}) can be obtained by means of the Gysin maps, up to a suitable choice of the cohomology extensions.

 In \S 7 we will prove that there is a $\pi$-ample divisor $h$
and a set of cohomology extensions for which the decomposition above is \textit{h-good}} (\cite[Definition 2.4.5]{DeC}). 
We recall that, in the paper \cite{DeC}, de Cataldo 
constructed five distingushed mixed-Hodge theoretic good splittings of the hypercohomology of the derived direct image  \cite[Theorem 1.1.1]{DeC}.
In general, the five good splittings turn out to be pairwise distinct but there is a natural condition, the existence of a $h$-good splitting
\cite[Definition 2.4.5]{DeC}, under which they in fact coincide \cite[Theorem 2.6.3]{DeC}. Such a condition says that the cup product with $h$, the 
first Chern class of a relatively ample line bundle, is \textit{homogeneous of degree two} (cf. \cite[\S 2]{DeC} and section 7). We will prove that this is indeed the case for the splitting provided by the Gysin morphism
if the $\pi$-ample divisor $h$ and the cohomology extensions are defined appropriately.

The last issue we are concerned with is addressed in \S 8. It is well known that, in the decomposition by supports of the perverse cohomology sheaves
$\sideset{_{}^{p}}{_{}^{i}}{\mathop \hc}(R\pi {_*}\bQ_{\tS}[n])\cong \oplus_r \sideset{_{}^{p}}{_{}^{i}}{\mathop \hc}(R\pi {_*}\bQ_{\tS}[n])_{\Delta_r}$, the summands are canonically defined.
Is it possible to {\it  identify  the canonical summands of } 
$\sideset{_{}^{p}}{_{}^{i}}{\mathop \hc}(R\pi {_*}\bQ_{\tS}[n])\mid _{\Delta _r^0}$?
In particular, is it possible to identify in the previous splitting the contribution of $IC^{\bullet}_{\Sc}\simeq R\xi {_*}\bQ_{\dc}[n]$, in such a way to compare the two standard resolutions of $\mathcal S$?
We will prove in Theorem \ref{LocStudy}  that the splitting
above
leads to a canonical decomposition of the cohomoloy spaces of $F_p$, where it is possible to ``recognize'' the summand coming from the fibre of the small resolution (cf. ((\ref{gammaAlow}), Theorem \ref{LocStudy}
and Remark \ref{recognize1}).

\vskip3mm

\textit{Acknowledgments} I would like to thank the referee for  carefully reading my manuscript and  for giving  a lot of advice that helped me to substantially improve the overall clarity of the paper.
\medskip

\section{Notations and Basic Facts}

\subsection{}

From now on, unless  otherwise stated, all cohomology and intersection cohomology groups are with
$\mathbb Q$-coefficients.

\medskip
\noindent
For a complex  algebraic variety $V$, we
denote by $H^{l}(V)$ and $IH^{l}(V)$ its cohomology and intersection
cohomology groups ($l\in\mathbb Z$). Let $D_c^b(V)$ be the constructible derived category  
 of sheaves of $\mathbb Q$-vector spaces on $V$. For
a complex of sheaves $\mathcal F^{\bullet}\in D_c^b(V)$, we denote by
$\mathbb H^l(\mathcal F^{\bullet})$ its hypercohomology groups. Let
$IC^{\bullet}_{V}$ denotes the intersection cohomology complex of $V$. If
$V$ is nonsingular, we have $IC^{\bullet}_{V}\cong \mathbb
Q_V[\dim_{\mathbb C} V]$, where $\mathbb Q_V$ is the
constant sheaf $\mathbb Q$ on $V$. 
We denote by $D^{\leq 0}(V)$, $D^{\geq 0}(V)$ ($D^{\leq 0}$, $D^{\geq 0}$ when no risk of confusion arises)
the t-structure on $D_c^b(V)$ associated with the middle perversity
\cite{BBD}, 
\cite[\S 10]{KS}, and by Perv(V) the abelian sub-category of perverse sheaves on $V$.  

A
resolution  of $V$ is a
projective surjective map $\pi:\widetilde V\to V$, such that
$\widetilde V$ is smooth and connected, and such that there exists an open dense
subset $U\subseteq V$ for which $\pi$ induces an isomorphism
 $\pi^{-1}(U)\to U$.

\medskip
\subsection{}
Fix integers $i,j,k,l$ such that:
$$
0< i< k \leq j< l \quad {\text {\rm and}} \quad r:=k-i< l-j=:c
$$
(the assumption $k \leq j$ is made for the sake of simplicity, in order to avoid to distinguish two cases in any statement of the paper \cite[\S 4]{DGF2strata}).
Let $\mathbb G_k(\mathbb C^l)$ denote the {\it Grassmann variety of} $k$-{\it planes in
$\mathbb C^l$}. Let $F\subseteq \mathbb C^l$ denote a fixed 
$j$-dimensional subspace. Define
$$
\mathcal S:=\left\{V\in \mathbb G_k(\mathbb C^l): \dim V\cap
F\geq i\right\}.
$$
$\mathcal S$ is called  {\it a single condition  Schubert variety or a special special Schubert variety}
\cite[p. 328]{CGM}.

\medskip
Consider the map \cite[p. 328]{CGM}:
$$
\pi: \widetilde{\mathcal S}\to  \mathcal S,
$$
where
$$
\widetilde{\Sc}:=\left\{(W,V)\in \mathbb G_i(F)\times
\mathbb G_k(\mathbb C^l): W\subseteq V\right\},
\quad \pi(W,V)=V\quad{\text{and}}\quad f(W,V)=W.
$$
The map $\pi$ is a resolution  of $\Sc$. We
have: 
$$
{\text{Sing}}(\Sc)=\left\{V\in \mathbb G_k(\mathbb C^l):
\dim V\cap F>i\right\}.
$$
and a Whitney stratification:
\begin{equation}\label{flag}
\Delta_1\subset\Delta_2\subset\dots\subset\Delta_r\subset\Delta_{r+1}=\mathcal S,\quad
r:=k-i.
\end{equation}
For this reason, in what follows we are going to call $\Sc  $
a {\it  special Schubert variety with $r+1 $ strata}.

In the stratification (\ref{flag}), $\Delta_p$ is   defined as
$$\Delta_p:=\left\{V\in \mathbb G_k(\mathbb C^l):
\dim V\cap F\geq k-p+1=:i_p\right\}.$$ So each $\Delta_p$ is a special Schubert variety as well.
In particular $\Delta_1$ is smooth, and
${\text{Sing}}\,\Delta_p= \Delta_{p-1}$ for all
$p\in\{2,\dots,r+1\}$. 
The map $\pi$
induces an isomorphism 
$$\pi^{-1}(\Sc^0)\cong \mathcal{S}^0, 
\quad \mathcal{S}^0:=\mathcal{S}\backslash{\text{Sing}}(\mathcal{S}).$$ 
Furthermore, every restriction
$$
\pi^{-1}(\Delta_p^0)\to
\Delta_p^0, \quad \Delta_p^0:=\Delta_p \backslash{\text{Sing}}(\Delta_p)=\Delta_p \backslash \Delta_{p-1}
$$
is a smooth and proper fibration on $\Delta_p^0$, with fibre
\begin{equation}\label{Fiber}
F_p:=\pi^{-1}(x)\cong \mathbb G_i(\mathbb C^{i_p}), \quad \forall x\in\Delta_p^0.
\end{equation}
So the stratification
(\ref{flag})
 is a
{\it stratification of} $\tS$ {\it adapted to} $\pi$ \cite{CGM},
\cite{Williamson}.

\medskip
\subsection{}
Denote by $S$ the {\it tautological  bundle} on $\bG_i(F)$ and  consider the following short exact sequence of vector
bundles on $\bG_i(F)$:
$$0 \to  S \to \bCl \to \kc \to 0,
$$
where  $\mathbb C^l $ is meant to be the trivial bundle of rank $l$ on $\bG_i(F)$ and  $\kc$ the cokernel of the bundle map $S \to \mathbb C^l$ 
(observe that  $\kc$ is NOT the universal quotient bundle on $\bG_i(F)$). 
\par 
Obviously, the projection $f:\tS \to \bG_i(F)$ allows us to identify $\tS$ with $\bG_{k-i} (\kc)$, the  {\it Grassmannian of
subspaces of dimension} $k-i$ {\it of the bundle} $\kc$. Furthermore, the projection
\par\noindent
$f:\tS\cong \bG_{k-i} (\kc)\to \bG_i(F)$ is a smooth and   proper fibration on $\bG_i(F)$, with fibre $\bG_{k-i}(\bC ^{l-i})$.

Also $\tS^0:=\pi^{-1}(\Sc^0)$ surjects on $\bG_i(F)$ via $f$
and we have the following fibre square commutative diagram
$$
\begin{array}{ccccc}
\tS^0 &\hookrightarrow  & \tS  \\
\stackrel{f^0}{}\downarrow &  &\stackrel {f}{}\downarrow  \\
\bG_i(F) & =  & \bG_i(F) \\
\end{array}.
$$
The leftmost map
$f^0: \tS^0\to \bG_i(F)$ is a (non-proper) fibration on $\bG_i(F)$, with fibre an open set of  $\bG_{k-i}(\bC ^{l-i})$.

\medskip
\subsection{}
As we have previously observed, any stratum $\Delta_p\subset \Sc$ is in turn a {\it single condition Schubert variety} so we have
a resolution of singularities, for all $2\leq p\leq r+1$
$$
\pi_p: \widetilde{\Delta }_p\to  \Delta _p,
$$
where
$$
\widetilde{\Delta }_p:=\left\{(Z,V)\in \mathbb G_{i_p}(F)\times
\mathbb G_k(\mathbb C^l): Z\subseteq V\right\},
\quad \pi_p(Z,V)=V
\quad{\text{and}}\quad f_p(Z,V)=Z
$$
(of course we have put $\tD_{r+1}=\tS$ and $\pi_{r+1}=\pi$).
Similarly as above,  we denote by $S_p:=S_{p , \bG_{i_p}(F)}$ the  tautological  bundle on $\bG_{i_p}(F)$, which is a bundle of rank $i_p=k-p+1$. If we consider the following short exact sequence of vector
bundles on $\bG_{i_p}(F)$:
$$0 \to  S_p \to \bCl \to \kc_p \to 0,
$$
the projection $f_p:\tD_p \to \bG_{i_p}(F)$ allows us to identify $\tD_p$ with $\bG_{p-1} (\kc_p)$, the  Grassmannian of
subspaces of dimension $p-1$ of the bundle $\kc_p$ on $\bG_{i_p}(F)$. Furthermore, we have the following commutative diagram:
\begin{equation}\label{D1}
\begin{array}{ccccccc}
 \nc_{p}^0 & \stackrel{\tilde\pi_{p}}{\cong}  & \pi^{-1}(\Delta_p^0) & & \\
\cap & & \cap & & \\
\nc_{p} & \stackrel{\tilde\pi_{p}}{\longrightarrow} & \overline{\Delta}_p & \stackrel{\tilde\iota_{p}}{\hookrightarrow} & \tS \\
\stackrel {\rho_{p}}{}\downarrow& &\stackrel {\pi}{}\downarrow & & \stackrel {\pi}{}\downarrow & \\
\widetilde\Delta_p &\stackrel{\pi_p}{\longrightarrow} & \Delta_p &\stackrel{\iota_{p}}{\hookrightarrow} & \Sc\\
 \cup & & \cup & & \\
 \widetilde\Delta_p^0 & \stackrel{\pi_p}{\cong}  & \Delta_p^0 & & \\
\end{array}
\end{equation}
where 
\begin{itemize}
\item $\overline{\Delta}_p:=\pi^{-1}(\Delta_p)$; 
\item $$\nc_{p}:=\left\{(W,Z,V)\in \mathbb F(i, i_p, F)\times
\mathbb G_k(\mathbb C^l): Z\subseteq V\right\}\cong  \bG_i(f^*_p(S_p));
$$
\item $F(i, i_p, F)$ is the variety of partial $(i,i_p)$-flags of $F$
and $\bG_i(f^*_p(S_P))$ is the Grassmannian of
subspaces of dimension $i$ of the bundle $f^*_p(S_p)$ on $\widetilde\Delta_p$;
\item the map $\rho_p$ is the natural projection
$$
\rho_p:\bG_i(f^*_p(S_p) _{\widetilde\Delta_p}) \to \widetilde\Delta_p
$$
with fibres $\bG_{i}(\bC ^{i_p})$;
\item $\nc_{p}^0:=\rho_p^{-1}(\widetilde\Delta_p^0)=\bG_i(f^*_p(S_p)\mid _{\widetilde\Delta_p^0})$.
\end{itemize}

This shows that {\it the fibration of equation} (\ref{Fiber}) {\it can be identified with the natural projection}
\begin{equation}
\label{Fiberimproved}
\rho_p:\nc_{p}^0=\bG_i(f^*_p(S_p)\mid _{\widetilde\Delta_p^0}) \to \widetilde\Delta_p^0
\end{equation}

\begin{remark}
\label{correspondence} By definition of $\Delta_p^0$, we have that  $\dim V\cap F=i_p$, $\forall V\in \Delta_p^0$. So we have:
$$ \tilde\pi_{p}:\nc_{p}^0\leftrightarrow \pi^{-1}(\Delta_p^0), \quad \tilde\pi_{p}(W,Z,V)=(W,V), \quad{\text{and}} \quad \tilde\pi_{p}^{-1}(W,V)=(W,V\cap F, V). $$
\end{remark}

\medskip
\subsection{} 

More generally, for all pairs $(p,q)$ such that $q<p\leq r+1$, we have an inclusion of special Schubert varieties $\Delta_q \subset \Delta_p$. Hence  we can extend everything we have said in the previous paragraph. We have
the following commutative diagram  
\begin{equation}\label{D2}
\begin{array}{ccccccc}
 \widetilde \Delta_{pq}^0 & \stackrel{\tilde\pi_{pq}}{\cong}  & \Delta_{pq}^0  & & \\
\cap & & \cap & & \\
\widetilde{\Delta}_{pq} & \stackrel{\tilde\pi_{pq}}{\longrightarrow} & \Delta_{pq} & \stackrel{\tilde\iota_{pq}}{\hookrightarrow} & \widetilde\Delta_p \\
\stackrel {\rho_{pq}}{}\downarrow& &\stackrel {\pi_{pq}}{}\downarrow & & \stackrel {\pi_{p}}{}\downarrow & \\
\widetilde\Delta_q &\stackrel{\pi_q}{\longrightarrow} & \Delta_q &\stackrel{\iota_{pq}}{\hookrightarrow} & \Delta_p  \\
 \cup & & \cup & & \\
 \widetilde\Delta_q^0 & \stackrel{\pi_q}{\cong}  & \Delta_q^0 & & \\
\end{array}
\end{equation}
where
\begin{itemize}
\item $\Delta_{pq}:=\pi_p^{-1}(\Delta_q)$; 
\item $\Delta_{pq}^0:=\pi_p^{-1}(\Delta_q^0)$;
\item $$\widetilde{\Delta}_{pq}:=\left\{(W,Z,V)\in \mathbb F(i_p, i_q, F)\times
\mathbb G_k(\mathbb C^l): Z\subseteq V\right\}\cong  \bG_{i_p}(f^*_q(S_q));
$$
\item $\widetilde{\Delta}_{pq}^0:=\rho_{pq}^{-1}(\widetilde\Delta_q^0)=\bG_{i_p}(f^*_q(S_q)\mid _{\widetilde\Delta_q^0})$; 
\item  the fibration $\Delta_{pq}^0 \to \Delta_{q}^0$ can be identified with the natural projection
\begin{equation}
\label{Fiberimprovedpq}
\widetilde{\Delta}_{pq}^0=\bG_{i_p}(f^*_q(S_q)\mid _{\widetilde\Delta_q^0}) \to \widetilde\Delta_q^0.
\end{equation}
with fibres
$$F_{pq}:=\bG_{i_p}(\bC^{i_q}).$$
\end{itemize}

\begin{remark}
\label{correspondenceimproved}
\begin{enumerate}
\item Needless to say, the properties of commutative diagram (\ref{D1}) can be deduced from (\ref{D2}) by setting $(p,q)=(r+1,p)$.
\item By definition of $\Delta_q^0$, we have that  $\dim V\cap F=i_q$, $\forall V\in \Delta_q^0$. So we have:
$$ \tilde\pi_{pq}:\widetilde \Delta_{pq}^0 \leftrightarrow \Delta_{pq}^0, \quad \tilde\pi_{pq}(W,Z,V)=(W,V), \quad{\text{and}} \quad \tilde\pi_{pq}^{-1}(W,V)=(W,V\cap F, V). $$
\end{enumerate}
\end{remark}

\medskip
\subsection{}
The following identities may be checked by  simple, sometimes tedious, calculations

\medskip
\begin{itemize}
\item $m_p:=\dim\Delta_p= \dim\tD_p=\dim\bG_{p-1} (\kc_{ p, \bG_{i_p(F)}})  =(k+1-p)(j+p-k-1)+(p-1)(l-k)$
\item $\Delta_{r+1}=\Sc$, \,\,\, $\tD_{r+1}=\tS$, \,\,\,   $n:=\dim\Sc=\dim\tS= m_{r+1}$
\item $k_{pq}:=\dim F_{pq}=\dim\bG_{i_p}(\bC^{i_q})=i_p(i_q-i_p)=(p-q)(k+1-p)$
\item $m_p-m_q=\dim\Delta_p-\dim\Delta_{q}= (p-q)(c+k+2-p-q)$
\item $d_{pq}:=m_p-m_q-k_{pq}=\dim\Delta_p-\dim\Delta_{pq}=(p-q)(c+1-q)$
\item $\delta_{pq}:=k_{pq}-d_{pq}=(p-q)(k-c+q-p)$ 
\end{itemize}
\medskip
\begin{remark}
\label{convention} If $k_{pq}<d_{pq}$ for all $ q<p$ (that is to say if $k\leq c$), then the resolution $\pi_p$ is small for any $p$ and we have $IC^{\bullet}_{\Delta_p}\cong R\pi_p {_*}\bQ_{\tD_p}[m_p]$ in $D_c^b(\Delta_p)$.
In this case the Decomposition Theorem is trivial for $\pi_p$. Hence, {\it from now on we assume} $c<k$. 
\end{remark}

\medskip

\section{Computation of the perverse cohomology sheaves}

This section is devoted to the determination of the perverse cohomology sheaves in (\ref{decomposition}).
The starting point of our analysis stems from the remark that the semisimple local systems involved in the decomposition are \textit{constant sheaves supported in the smooth 
locus of the strata} $\Delta _p^0=\Delta _p \backslash \Delta_{p-1}$. In other words, the decomposition (\ref{decomposition})  takes the following form

\begin{equation}\label{decompositionimproved}
R\pi {_*}\bQ_{\tS}\cong \bigoplus_{p, r}IC_{\Delta_p}^{\bullet}[r]^{\oplus m_{pr}}
\end{equation}
for suitable multiplicities $m_{pr}\in \mathbb N_0$. As we are going to see, this follows just combining the Decomposition Theorem with  the global invariant cycle theorem, taking also into account of the
resolution of the strata described in \S 2.4.

To this end, consider a fibre square commutative diagram:
$$
\begin{array}{ccccccc}
\widetilde \Delta&\stackrel
{\jmath}{\hookrightarrow}&\widetilde X\\
\stackrel {\rho}{}\downarrow & &\stackrel
{\pi}{}\downarrow \\
\Delta&\stackrel
{\imath}{\hookrightarrow}&X\\
\end{array}
$$
where $\pi$ is a resolution, $\Delta $ is a locally closed subset of $X$ and $\rho$ is a smooth and proper fibration.
Combining the  Base Change property \cite[p. 322]{Iversen} with \cite[Remarks 1.5.1 and 1.6.2 (3)]{DeCMReview}, we get the following  isomorphisms (the second one is not uniquely determined), that will be extensively used in the sequel
\begin{equation}
\label{fibration} 
 {\imath}{^*}R\pi{_*}\bQ _{\widetilde X}\cong R\rho{_*}\bQ _{\widetilde{\Delta}}\cong \oplus_{i\in \bZ}R^i\rho{_*}\bQ _{\widetilde \Delta}[-i]
\quad{\text{in}}\quad D^b_c(\Delta).
\end{equation}

\medskip

\begin{remark}
\label{NewRemark}
\begin{enumerate}
\item By Deligne's theorem, $R^i\rho{_*}\bQ _{\widetilde \Delta}$ is a local system so equation (\ref{fibration}) implies that 
${\imath}{^*}R\pi{_*}\bQ _{\widetilde X}$ is a {\it direct sum of shifted local system in}   $D_c^b(\Delta)$. 
As in  \cite[Remark 1.5.1]{DeCMReview}, we also have:
\begin{equation}\label{compHR}
\sideset{_{}^{p}}{_{}^{m}}{\mathop \hc}({\imath}{^*}R\pi{_*}\bQ _{\widetilde X})\cong R^{m-\dim \Delta}\rho{_*}\bQ _{\widetilde \Delta}[\dim \Delta] \quad{\text and}\quad
\sideset{_{}^{p}}{_{\leq k}^{}}{\mathop \tau}({\imath}{^*}R\pi{_*}\bQ _{\widetilde X})\cong \oplus_{i=dim \Delta}^kR^i\rho{_*}\bQ _{\widetilde \Delta}[-i].
\end{equation}
\item Assume now that $\rho$ is a locally trivial fibration with fibre $F$ and that
 the restriction map $H^{\alpha}(\tD)\to
H^{\alpha}(F)$ is onto for all $\alpha\in\mathbb Z$.
Under this assumption, the global invariant cycle theorem \cite[\S 4.3.1]{VoisinII}, \cite[Theorem 1.2.4]{DeCMReview}: 
$$H^{\alpha}(F)^{\pi _1 (\Delta)}= Im (H^{\alpha}(\tD) \longrightarrow H^{\alpha}(F)),  $$
implies that  the fundamental group of $\Delta $ acts trivially on the cohomology of $F$. 
Hence, the splitting (\ref{compHR}) takes the following form:
\begin{equation}\label{senzaLH}
 R\rho{_*}\bQ _{\widetilde{\Delta}}\cong\oplus_{\alpha=0}^{2\dim F} H^{\alpha}(F)\otimes \mathbb Q_{\Delta}[-\alpha],
\, \, \,\, \sideset{_{}^{p}}{_{}^{m}}{\mathop \hc}({\imath}{^*}R\pi{_*}\bQ _{\widetilde X})\cong H^{m-\dim \Delta}(F)\otimes \bQ _{ \Delta}[\dim \Delta].
\end{equation}
\noindent
In other words {\it the  local system involved in} (\ref{compHR})  {\it are constant sheaves on} $\Delta$.
\item We  observe that the hypothesis of previous remark is satisfied by the resolutions described in \S 2. In the meanwhile, we fix some notations.
For all $\beta\in\mathbb Z$, set
$A^{\beta}_{pq}:=H^{\beta}(F_{pq})=H^{\beta}(\bG_{i_p}(\bC ^{i_q}))$, and
$a^{\beta}_{pq}:=\dim_{\mathbb Q} H^{\beta}(F_{pq})$. 
By \cite[Corollary 3.2.4]{Manivel}, $a^{\beta}_{pq}=0$ for $\beta $ odd and $A^{2\mid \lambda \mid}_{pq}$ has a canonical basis which is in one-to-one correspondence with all partitions $\lambda$ contained in a 
$(p-q)\times i_p$ rectangle. 

Recall from \S 2.5 that $F_{pq}$ coincides with any fibre of the fibration (\ref{Fiberimprovedpq}):  
$$\widetilde{\Delta}_{pq}^0=\bG_{i_p}(f^*_q(S_q)\mid _{\widetilde\Delta_q^0}) \to \widetilde\Delta_q^0\subset \bG_{q-1} (\kc_{q, \, \bG_{i_q}(F)})$$
and obviously we have $F_{pq}\subset \bG_{i_p}(F)$, as a Grassmann submanifold, hence $H^{\bullet}(F_{pq})$ is canonically contained in $H^{\bullet}(\bG_{i_p}(F))$ and
the hypothesis of remark  (2) is thus satisfied via pull-back from $\bG_{i_p}(F)$. 
\end{enumerate}
\end{remark} 

By Remark \ref{NewRemark}, 
the semisimple local systems involved in  (\ref{decomposition})  are constant sheaves supported in the smooth part of the strata,
hence the Decomposition Theorem implies that the derived direct image $R\pi {_*}\bQ_{\tS}$ splits according to (\ref{decompositionimproved}).
So we are left with the computation of the occurring multiplicities $m_{pr}$. This  will be achieved in Theorem \ref{PerverseCohomology} (by downward induction on $p$)
by means of a combinatoric argument based on dimension of the stalk cohomology groups
\begin{equation}
\label{dimstalk}
\hc^{\bullet}(R\pi {_*}\bQ_{\tS})_y \cong H^{\bullet}(\pi ^{-1}(y)) \quad{\text{and}} 
\end{equation}
$$\hc^{\bullet}(IC_{\Delta_p}^{\bullet})_y \cong H^{\bullet + \dim \Delta_p}(\xi_p ^{-1}(y)), $$
where $ \xi _p: \dc _p \to \Delta_p$ denotes the small resolution of the stratum $\Delta_p$ (compare with the proof of (\ref{ingredients})).

On that basis, the proof of Theorem \ref{PerverseCohomology}
reduces to a dimensional counting (compare with formulas (\ref{check}) and (\ref{thesis})).
We record the necessary ingredients in the following Lemma.

\begin{lemma}\label{ingredients} For any pair $q< p$ and for any $y\in \Delta_q^0$, we have:
$$\dim (\hc^{\beta}(R\,\pi_{p *}\bQ_{\tD _p})_y) =  \dim( H^{\beta}(\bG_{i_p}(\bC ^{i_q})))=: a^{\beta}_{pq} \quad \quad{\text{and}} $$
$$\dim (\hc^{\beta}(IC_{\Delta_p}^{\bullet})_y) =  \dim( H^{\beta+ \dim \Delta_p}(\bG_{p-q}(\mathbb C^{c-q+1})))=:b^{\beta}_{pq}.$$
\end{lemma}
\begin{proof}
Recall from \S 2.5 that the morphism $\pi_p^{-1}(\Delta _q^0) \to \Delta_q^0$  is a smooth fibration with
fibre $F_{pq}=\bG_{i_p}(\bC ^{i_q})$. Then the first formulas follows just combining  (\ref{fibration}) and (\ref{senzaLH}).

As for the second, we need to define the small resolution of the stratum $\Delta _p$.
Consider the map:
$$
\xi_p: \dc _p \to \Delta_p,
$$
where
$$
\dc_p:=\left\{(V,U)\in \mathbb G_k(\bC^l)\times
\mathbb G_{k+j-i_p}(\mathbb C^l): V+F\subseteq U\right\},
\quad \xi_p(V,U)=V,
$$
which is another standard resolution of $\Delta_p$. Similarly as in \S 2.2,
every restriction
$$
\xi^{-1}_p(\Delta_q^0)\to
\Delta_q^0, 
$$
is a smooth and proper fibration on $\Delta_q^0$, with fibre
$$
G_{pq}:=\xi_p ^{-1}(V) \cong \left\{U\in 
\mathbb G_{k+j-i_p}(\mathbb C^l): V+F\subseteq U\right\}  \cong \mathbb G_{p-q}(\mathbb C^{c-q+1}), \quad \forall V\in\Delta_q^0.
$$
Since $\overline{k}_{pq}:= \dim G_{pq}=(p-q)(c-p+1)$ and $m_p-m_q=\dim\Delta_p-\dim\Delta_{q}= (p-q)(c+k+2-p-q)$ (compare with \S 2.6),
then $\overline{k}_{pq}=(p-q)(c-p+1)<m_p-m_q- \overline{k}_{pq}=(p-q) (k-q+1)$, because in \ref{convention} we assumed $c<k$.
So, {\it  the resolution} $\xi_p$ {\it is small} and we have 
\begin{equation}
\label{icsmall}
IC^{\bullet}_{\Delta_p}\cong R\xi_p {_*}\bQ_{\dc _p}[m_p]\quad \text{\rm in }\quad D_c^b(\Delta_p).
\end{equation}
Finally, by (\ref{fibration}) and (\ref{senzaLH}) we have
$$\hc^{\beta}(IC_{\Delta_p}^{\bullet})_y\cong H^{\beta+ \dim \Delta_p}(\xi_p ^{-1}(y))= H^{\beta+ \dim \Delta_p}(G_{pq})\cong H^{\beta+ \dim \Delta_p}(\bG_{p-q}(\mathbb C^{c-q+1}))
$$
and we are done.
\end{proof}

\begin{remark}\label{iclocallytrivial}
By (\ref{icsmall}), we have 
\begin{equation}
IC^{\bullet}_{\Delta_p}\mid_{\Delta^0_q}\cong  R\xi_p {_*}\bQ_{\dc _p}[m_p]\mid_{\Delta^0_q}
\cong \oplus_{i\in \bZ}R^i\xi_p{_*}\bQ _{\dc _p}[m_p-i]\mid_{\Delta^0_q}
\end{equation}
and each $R^i\xi_p{_*}\bQ _{\dc _p}\mid_{\Delta^0_q}$ {\it is a  local system with fibres}  $B^i_{pq}:=H^i(G_{pq})=H^i(\mathbb G_{p-q}(\mathbb C^{c-q+1}))$.
\end{remark}

\begin{notations}\label{Dpq}
For future reference, we now define an useful subspace $D^{2\mid \lambda \mid}_{pq}\subset A^{2\mid \lambda \mid}_{pq}$ (cf. Remark \ref{NewRemark} (3)). If $i_p<c+1-q$ (i.e. if $\delta_{pq}<0$) then we set $D^{2\mid \lambda \mid}_{pq}=0$. On the contrary, 
for $\mid \lambda \mid \leq (p-q)\times (k-c+q-p)=\delta_{pq} $ we denote by $D^{2\mid \lambda \mid}_{pq}\subset A^{2\mid \lambda \mid}_{pq}$ 
the subspace spanned by partitions $\lambda =(\lambda_1 , \dots ,\lambda _{p-q} )$ such that
$\lambda _{i}\leq (k-c+q-p)$, $1\leq i \leq p-q$. Observe that $D^{2\mid \lambda \mid}_{pq}$ has a canonical basis which is in one-to-one correspondence with all partitions $\lambda$ contained in a 
$(p-q)\times (k-c+q-p)$ rectangle. In other words, $D^{2\mid \lambda \mid}_{pq}$ can be identified the cohomology of a Grassmann submanifold 
\begin{equation}
\label{subgrassmannian}
D^{2\mid \lambda \mid}_{pq} \cong H^{2\mid \lambda \mid}(T_{pq}), \quad T_{pq}:=\bG _{p-q}(\bC ^{k-c})\subset\bG _{p-q}(\bC ^{i_q})= F_{pq}.
\end{equation}
Since
$\dim T_{pq}=(p-q)\cdot(k-c+q-p)= \delta_{pq}$ (compare with \S 2.6), by Hard Lefschetz, we have 
\begin{equation}\label{Hard}
D^{\delta_{pq} -\alpha}_{pq}\cong D^{\delta_{pq} +\alpha}_{pq}.
\end{equation}
We set $d^{\alpha}_{pq}:= \dim D^{\alpha}_{pq}$.
\end{notations}

\medskip

\begin{theorem}
\label{PerverseCohomology}  
We have
$$\sideset{_{}^{p}}{_{}^{i}}{\mathop \hc}(R\,\pi_{p *}\mathbb Q_{\tD_{p}}[m_{p}]) \cong  \bigoplus_{q=0}^p   D_{pq }^{\delta_{pq} + i}\otimes_{\mathbb Q}
R{\iota_{pq}}_* IC_{\Delta_q}^{\bullet}.
$$
\end{theorem}

\begin{proof}
We need  to prove that 
$$\sideset{_{}^{p}}{_{}^{i}}{\mathop \hc}(R\,\pi_{p *}\mathbb Q_{\tD_{p}}[m_{p}])_{\Delta_q} \cong     D_{pq }^{\delta_{pq} + i}\otimes_{\mathbb Q}
R{\iota_{pq}}_* IC_{\Delta_q}^{\bullet},
$$
where 
$\sideset{_{}^{p}}{_{}^{i}}{\mathop \hc}(R\,\pi_{p *}\mathbb Q_{\tD_{p}}[m_{p}])_{\Delta_q}$ denotes the  ($\Delta_q$)-summand in the decomposition by supports of 
$\sideset{_{}^{p}}{_{}^{i}}{\mathop \hc}(R\,\pi_{p *}\mathbb Q_{\tD_{p}}[m_{p}])$ \cite[\S 1.1]{DeC}.
The last statement is a consequence of the following 

{\it Claim:
In} $Perv(\Delta_{p}\backslash \Delta_l)$ {\it we have the isomorphisms}  
$$\sideset{_{}^{p}}{_{}^{i}}{\mathop \hc}(R\,\pi_{p *}\mathbb Q_{\tD_{p}}[m_{p}])_{\Delta_q\backslash \Delta_l} \cong     D_{pq }^{\delta_{pq} + i}\otimes_{\mathbb Q}
R{\iota_{pq}}_* IC_{\Delta_q\backslash \Delta_l}^{\bullet},
$$
{\it for all} $0\leq l< q\leq p$.

If $q-l\geq 2$,  the perverse sheaf 
$\sideset{_{}^{p}}{_{}^{i}}{\mathop \hc}(R\,\pi_{p *}\mathbb Q_{\tD_{p}}[m_{p}])_{\Delta_q\backslash \Delta_l}$ is the {\it intermediate extension} 
of 
$\sideset{_{}^{p}}{_{}^{i}}{\mathop \hc}(R\,\pi_{p *}\mathbb Q_{\tD_{p}}[m_{p}])_{\Delta_q\backslash \Delta_{l+1}}$  to $\Delta_q\backslash \Delta_l$
\cite[Proposition 2.1.9]{BBD}, \cite[Theorem 5.2.12, Theorem 5.4.1]{Dimca2}, \cite[\S 2.7]{DeCMReview}. 
So we can assume $l+1=q\leq p$. 

The proof proceeds now by induction on $p-q$. Since the  starting case $q=p$ is trivial,
we can assume $l+1=q<p$, and 
we only need to prove:
\begin{equation}\label{toprove}
\sideset{_{}^{p}}{_{}^{i}}{\mathop \hc}(R\,\pi_{p *}\mathbb Q_{\tD_{p}}[m_{p}])_{\Delta_q^0} \cong     D_{pq }^{\delta_{pq} + i}\otimes
R{\iota_{pq}}_* \bQ_{\Delta_q^0}[m_q].
\end{equation}

By Remark (\ref{NewRemark}), we have
\begin{equation}
\label{LHThm}
R\,\pi_{p *}\mathbb Q_{\tD_{p}}[m_{p}]\mid_{\Delta_q^0}\cong
\oplus_{\alpha=0}^{2dimF_{pq}} A^{\alpha}_{pq}\otimes \mathbb Q_{\Delta_q^0}[-\alpha+m_p]
\end{equation}

On the other hand, the Decomposition Theorem (with proper base change) says that 
\begin{equation}\label{DecThmProof}
R\,\pi_{p *}\mathbb Q_{\tD_{p}}[m_{p}]\mid_{\Delta_q^0}\cong
\oplus_{i\in \bZ}\sideset{_{}^{p}}{_{}^{i}}{\mathop \hc}(R\,\pi_{p *}\mathbb Q_{\tD_{p}}[m_{p}])_{\Delta_q^0}[-i]
 \bigoplus_{i\in \bZ , \, r> q}^{r\leq p}\sideset{_{}^{p}}{_{}^{i}}{\mathop \hc}(R\,\pi_{p *}\mathbb Q_{\tD_{p}}[m_{p}])_{\Delta_r}[-i]\mid_{\Delta_q^0}.
\end{equation}
Furthermore,  arguing by induction and recalling that ${\iota_{pr}}^*\circ {\iota_{pr}}_*=1$ \cite[p. 110]{Iversen},
we get 
\par\noindent
($\forall$ $q<r<p$)
\begin{equation}\label{InductionIC}
\sideset{_{}^{p}}{_{}^{i}}{\mathop \hc}(R\,\pi_{p *}\mathbb Q_{\tD_{p}}[m_{p}])_{\Delta_r\backslash \Delta_{l}}\mid_{\Delta_q^0}\cong 
 D_{pr }^{\delta_{pr} + i}\otimes
R{\iota_{pr}}_* IC_{\Delta_r}^{\bullet}\mid_{\Delta_q^0}\cong 
D_{pr }^{\delta_{pr} + i}\otimes
(\oplus_{\alpha} B^{\alpha+m_r}_{rq} \otimes \bQ _{\Delta^0_q}[-\alpha]).
\end{equation}

Combining  (\ref{LHThm}), (\ref{DecThmProof}) and (\ref{InductionIC})  and taking account of (\ref{compHR})
we infer
\begin{equation}\label{beforeconclusion}
A^{j+m_p}_{pq}\otimes \mathbb Q_{\Delta^0_q}\cong \sideset{_{}^{p}}{_{}^{j+m_q}}{\mathop \hc}(R\,\pi_{p *}\mathbb Q_{\tD_{p}}[m_{p}])_{\Delta_q^0}\oplus
R^{j}(IC^{\bullet}_{\Delta_p} \mid_{\Delta_q^0}) \oplus
\bigoplus_{ r> q,\, \beta+\gamma =j}^{r< p}D_{pr }^{\delta_{pr} + \beta}\otimes  B^{\gamma+m_r}_{rq} \otimes \bQ _{\Delta^0_q}.
\end{equation}
By  Hard Lefschetz and (\ref{Hard}), 
it suffices  to prove (\ref{toprove}) for every $i\geq 0$.
As $IC^{\bullet}_{\Delta_p}$ satisfies the support condition, 
$R^{i}(IC^{\bullet}_{\Delta_p} \mid_{\Delta_q^0})=0$, for $i\geq -m_q$, and (\ref{beforeconclusion}) become

\begin{equation}\label{conclusion}
A^{j+m_p}_{pq}\otimes \mathbb Q_{\Delta^0_q}\cong \sideset{_{}^{p}}{_{}^{j+m_q}}{\mathop \hc}(R\,\pi_{p *}\mathbb Q_{\tD_{p}}[m_{p}])_{\Delta_q^0}
 \oplus
\bigoplus_{ r> q,\, \beta+\gamma =j}^{r< p}D_{pr }^{\delta_{pr} + \beta}\otimes  B^{\gamma+m_r}_{rq} \otimes \bQ _{\Delta^0_q}
\end{equation}
as  trivial local systems on $\Delta^0_q$.
So the only thing we need to check is that
\begin{equation}
\label{check}
a^{j+m_p}_{pq} =d_{pq }^{\delta_{pq} + j+m_q} +\sum_{  r= q+1}^{p-1}\sum_{\, \beta+\gamma =j}d_{pr }^{\delta_{pr} + \beta}\cdot b^{\gamma+m_r}_{rq},
\end{equation} 
 for every $j\geq -m_q$.
Recalling that $m_p-m_q-\delta_{pq}=2d_{pq}$, it is enough to prove 
\begin{equation}\label{thesis}
a^{s}_{pq} =d_{pq }^{s-2d_{pq}} +\sum_{  r= q+1}^{p-1}\sum_{ \, \beta+\gamma =s}d_{pr }^{ \beta-2d_{pr}}\cdot b^{\gamma}_{rq}
\end{equation}
for every $s\geq m_p-m_q$ ($\geq 2d_{pq}$).

To this end, we define a set of linear injective maps $\tilde{\gamma}^q_{pr}:D_{pr }^{ \beta}\otimes B_{rq}^{\gamma}
\to A^{\beta + \gamma + 2d_{pr}}_{pq}$, for any $q<r<p$.
If $\sigma_{\lambda}\in B_{rq}^{\gamma}$ and $\sigma_{\mu}\in D^{\beta}_{pr}$, with $\lambda:=(\lambda_1, \dots , \lambda_{r-q})$ and $\mu:=(\mu_1, \dots , \mu_{p-r})$  we define
$$
\tilde{\gamma}^q_{pr}(\sigma_{\mu}\otimes \sigma_{\lambda} ):= \sigma_{\nu}, \,\,\,\,\,\, \nu=(\nu_1, \dots , \nu_{p-q}):=(\mu_1+c+1-r, \dots , \mu_{p-r}+c+1-r, \lambda_1 , \dots , \lambda_{r-q}).
$$

Taking into account the Definitions of $D^{\alpha}_{pr}$ and $B_{rq}^{\beta}$ (compare with Notations \ref{locsys}),   we find that
\begin{equation}
\label{gammalpq}
{\text{\it the image \quad $\tilde{\gamma}^q_{pr}(D_{pr }^{ \beta}\otimes B_{rq}^{\gamma})\subset A^{\beta + \gamma + 2d_{pr}}_{pq}$ is spanned by the cohomology classes \,\,   
$\sigma_{\nu} $    }}
\end{equation}
$$
{\text{\it such that  \quad    $\nu_1\geq c+1-r$, $\dots $, $\nu_{p-r}\geq c+1-r$, \quad $\nu_{p-r+1}\leq c+1-r$, $\dots $, $\nu_{p-q}\leq c+1-r$}}. 
$$
Similarly, we define
$\tilde{\gamma}^0_{pq}:D_{pq }^{ \beta}\to A^{\beta  + 2d_{pq}}_{pq}:$ 
$$
 \tilde{\gamma}^0_{pq}(\sigma_{\mu} ):= \sigma_{\nu}, \,\,\, \nu:=(\mu_1+c+1-q, \dots , \mu_{p-q}+c+1-q)
$$
and we have
\begin{equation}
\label{gamma0pq}
{\text{\it the image \quad $\tilde{\gamma}^0_{pq}(D_{pq }^{ \beta})\subset A^{\beta  + 2d_{pq}}_{pq}$ is spanned by the cohomology classes \,\,  
$\sigma_{\nu} $    }}
\end{equation}
$$
{\text{\it such that  \quad    $\nu_1\geq c+1-q$, $\dots $, $\nu_{p-q}\geq c+1-q$. }}
$$
Combining (\ref{gammalpq}) with (\ref{gamma0pq}) we conclude that
\begin{equation}
\label{gammaAhigh}
A^{s}_{pq} \cong \bigoplus_{r=q+1}^{p-1}\bigoplus_{ \beta + \gamma =s} \left( \tilde{\gamma}^q_{pr}(D_{pr }^{ \beta-2d_{pr}}\otimes B_{rq}^{\gamma}) \right) \oplus \tilde{\gamma}^0_{pq}
(D_{pq }^{s-2d_{pq}}), \quad {\text {\rm  if}} \quad s \geq 2\overline{k}_{pq}+1
\end{equation}
and
\begin{equation}
\label{gammaAlow}
A^{s}_{pq} \cong \bigoplus_{r=q+1}^{p-1}\bigoplus_{ \beta + \gamma =s} \left( \tilde{\gamma}^q_{pr}(D_{pr }^{ \beta-2d_{pr}}\otimes B_{rq}^{\gamma}) \right) \oplus \tilde{\gamma}^0_{pq}
(D_{pq }^{s-2d_{pq}})\oplus B^{s}_{pq}, \quad {\text {\rm  if}} \quad s \leq 2\overline{k}_{pq}
\end{equation}
Then (\ref{thesis}) follows from (\ref{gammaAhigh}), because $\overline{k}_{pq}< d_{pq}$, and we are done.
\end{proof}

\medskip

\begin{remark}
 In Proposition \ref{gamma} and Theorem \ref{LocStudy} we give a geometric interpretation of the linear maps $\tilde{\gamma}^q_{pr}$, introduced in the proof of Theorem \ref{PerverseCohomology}, and of
the splittings (\ref{gammaAhigh}) and (\ref{gammaAlow}).
\end{remark}

\medskip

\section{The normal bundle of $\tDpq^0\subset \tD_p$}

In this section we are going to compute the normal bundle of the smooth subvariety  $\tDpq^0$ of $ \tD_p$ (compare with \S 2.5).
In \S  2.4 we denoted by  $S_q$ the tautological bundle of $\bG_{i_q}(F)$ (whose rank is $i_q=k-q+1$) and by 
$\kc_q$ the cokernel of the bundle map 
$S_q \to \mathbb C^l $ on $  \bG_{i_p}(F)$. Similarly, we denote by
$S_k$ the the tautological bundle   of $\bG_{k}(\bC^l)$.

By abuse of notation, we still denote by $S_p$, $S_q$, $S_k$ and $\mathcal K_q$  their  pull-back on $\tDpq $.
 
By Remark \ref{correspondenceimproved}, the bundles $S_p$, $S_q$, $S_k$ and $\mathcal K_q$ are also well defined  on $\tDpq^0\cong \Delta_{pq}^0$. Furthermore, again by
Remark \ref{correspondenceimproved}, we have
$$S_q\equiv S_k \cap F \quad{\text{in}} \quad  \Delta_{pq}^0,$$ 
So, also 
$S_k + F$ and $\mathcal J_{pq} :=\mathbb C^l\slash (S_{k} + F)$ are well defined vector bundles over
$\widetilde{\Delta }_{pq}^0\simeq \Delta _{pq}^0$. We are now able to compute the normal bundle of $\tDpq^0$ in $\tD_p$:

\medskip
\begin{proposition}\label{normalbundle} The normal bundle $\mathcal N_{pq}$  of the smooth subvariety $\tDpq^0$ in $\tD_p$ is 
$$\mathcal N_{pq}\cong {\text{Hom}}(S_q\slash S_p, \mathcal J_{pq}).$$
\end{proposition}

\begin{proof}
Since both $\Delta _{pq}^0 $ and $\widetilde{\Delta }_{p}$ fiber over $\bG:= \bG _{i_p}(F)$, the normal bundle $\mathcal N_{pq}$ of $\Delta _{pq}^0 \subset \widetilde{\Delta }_{p}$
coincides with the relative normal bundle with respect to $\bG$.

The relative tangent bundles of $\widetilde{\Delta }_{p}$ is \cite[p. 435, B.5.8]{FultonIT}:
$$
 T_{\widetilde {\Delta }_{p}\slash \mathbb G}\cong
{\text{Hom}}(S_k\slash S_{p}, \bC ^l\slash S_{k}),
$$
and the relative tangent bundles of $\Delta _{pq}^0 $ fits in the following short exact sequence:

$$
0\to {\text{Hom}}(S_k\slash S_{q}, \bC ^l\slash S_{k}) \to T_{\Delta _{pq}^0 \slash \mathbb G} \to {\text{Hom}}(S_{q}\slash S_{p}, F\slash S_{q}) \to 0.
$$
We have the following commutative diagram of vector bundles on $\bG$:

$$
\begin{array}{ccccccccc}
  &     &  0                                                &     &              0                      & \to & {\text{Hom}}(S_{q}\slash S_{p}, F\slash S_{q}) &  \to    &  \dots  \\
  &     &  \downarrow                                       &     &              \downarrow             &     &         \updownarrow                           &     &   \\
0 & \to & {\text{Hom}}(S_k\slash S_{q}, \bC ^l\slash S_{k}) & \to & T_{\Delta _{pq}^0 \slash \mathbb G} & \to & {\text{Hom}}(S_{q}\slash S_{p}, F\slash S_{q}) & \to & 0 \\
 &     &  \downarrow                                       &     &              \downarrow             &     &         \downarrow                           &     &   \\
0 & \to & {\text{Hom}}(S_k\slash S_{p}, \bC ^l\slash S_{k}) & \rightarrow & T_{\widetilde {\Delta }_{p}\slash \mathbb G}\mid_{\Delta _{pq}^0 } & \to & 0 &  &  \\
&     &  \downarrow                                       &     &              \downarrow             &     &         \downarrow                           &     &   \\
\dots  & \to & {\text{Hom}}(S_q\slash S_{p}, \bC ^l\slash S_{k}) & \to & \nc_{pq} & \to & 0 &  &  \\
\end{array}
$$     
where the leftmost column arise by applying ${\text{Hom}}(\cdot , \bC ^l\slash S_{k})$ to 
$$
0\to S_{q}\slash S_{p} \to S_k\slash S_{p} \to S_k\slash S_{q} \to 0.
$$
So we find 
$$
0\to {\text{Hom}}(S_{q}\slash S_{p}, F\slash S_{q}) \to {\text{Hom}}(S_{q}\slash S_{p}, \bC ^l \slash S_{k})  \to \nc_{pq} \to 0.
$$
and conclude in view of 
$$0 \to F\slash S_{q} \to  \bC ^l \slash S_{k} \to  \mathbb C^l\slash (S_{k} + F)= \mathcal J_{pq}\to 0.$$
\end{proof}

On the general fiber $F_{pq}=\bG_{i_p}(\bC^{i_q})$ of
the fibration $\pi_{pq}: \Delta_{pq}^0 \to \Delta_{q}^0$ (compare with (\ref{Fiberimprovedpq})),
both $S_q$ and $\mathcal J_{pq}$  are constant vector bundles with ranks $i_q$ and  $c+1-q$ respectively.
Hence, Proposition \ref{normalbundle} implies immediately

\begin{corollary}\label{cupc} Let us denote by $\qc_{\bG_{i_p}(\bC^{i_q})}$ the universal quotient bundle of $F_{pq}=\bG_{i_p}(\bC^{i_q})$ and by
$c_{pq}:=c_{p-q}(\qc_{\bG_{i_p}(\bC^{i_q})})$ the top Chern class of $\qc_{\bG_{i_p}(\bC^{i_q})}$.
Then we have: 
\begin{enumerate}
\item $$\mathcal N_{pq}\mid_{F_{pq}}\cong \qc_{\bG_{i_p}(\bC^{i_q})}^{\vee}\otimes \bC^{c+1-q}.$$
\item The top Chern class of $\mathcal N_{pq}\mid_{F_{pq}}$ is
$$c_{top}(\nc_{pq}\mid_{F_{pq}})=((-1)^{p-q}c_{p-q}(\qc_{\bG_{i_p}(\bC^{i_q})}))^{c+1-q}=\pm c_{pq}^{c+1-q}.$$
\end{enumerate}
\end{corollary}

\medskip

\section{Gysin morphisms}

Our main aim in this section is to provide an explicit decomposition  (\ref{decomposition}).  
Consider the commutative diagram (\ref{D2}) and set $\jmath_{pq} :=\tilde\iota_{pq}\circ \tilde\pi_{pq}$.
Since $\tDpq $  is smooth,   the map $\jmath_{pq}$   
induces   a Gysin morphism 
$R{\jmath_{pq}}{_*}\mathbb Q_{\tDpq}\to \mathbb Q_{\tD_p}[2d_{pq}]$. We are going to prove that 
it is possible to describe explicitly the decomposition (\ref{decomposition}) by means of the Gysin morphism above and 
the Leray-Hirsch Theorem. Such a description is NOT canonical as it depends on the choice 
of a {\it cohomology extension} (cf. Remark \ref{cohomologyextension}). Further, we will describe in Remark \ref{differentcohext} the effect that  a change of the cohomology extension has on the splitting.
The main point of the proof consists in showing that the Gysin morphism provides an  isomorphism

\begin{equation}\label{canonical}
\sideset{_{}^{p}}{_{}^{i}}{\mathop \hc}(R\,\pi_{p *}\mathbb Q_{\tD_{p}}[m_{p}])_{\Delta_q}
\cong D_{pq}^{\delta_{pq}+i}\otimes_{\mathbb Q}
R\iota_{pq}IC_{\Delta_q}^{\bullet}
\end{equation}
which turn out to be \textit{canonical, namely independent of the cohomology extension} (cf. Theorem \ref{isoDE}).
We also show in passing that the geometrical meaning of the isomorphism (\ref{canonical}) is that the perverse cohomology sheaves
are generated (via the Gysin morphism) by
 the cohomology classes of the fibre  that are ``divisible'' by the top Chern class of the normal bundle  described in \S 4.
 (cf. Remark \ref{commento}).

\bigskip

Let us come back to the fibre square commutative diagram  of \S 3.
Since $\tD $  is smooth,  the map $\jmath$ also  
induces  a Gysin morphism 
$R{\jmath{_*}}\mathbb Q_{\tD}\to \mathbb Q_{X}[2d]$, $d:= dimX -dim\tD$ (cf. \cite[p. 83]{FultonCF}).
By the self-intersection formula \cite[p. 92]{FultonCF}, \cite{Jou}, \cite{Yama}, 
 the composite of such a Gysin morphism with the pull-back morphism $\mathbb Q_{X}\to R{\jmath{_*}}\mathbb Q_{\tD}$
coincides with the  cup product with the top Chern class  $\kappa$ of the normal bundle of $\tD \subset X$:

\begin{equation}
\label{cupwith}
 \cup \kappa : R{\jmath{_*}}\mathbb Q_{\tD}\to R{\jmath{_*}}\mathbb Q_{\tD}[2d].
\end{equation}
By composition with $\rho$, the  morphism above
induces 
$
\gamma: R{\rho}{_*}\mathbb  Q_{\tD}\to  R{\rho}{_*}\mathbb Q_{\tD}[2d].
$

\begin{remark}\label{cohomologyextension} Similarly as in Remark \ref{NewRemark} (2),
assume  that $\rho$ is a locally trivial fibration with fibre $F$ and that
 the restriction map $H^{\alpha}(\tD)\to
H^{\alpha}(F)$ is onto for all $\alpha\in\mathbb Z$.
In this case, the trivialization (\ref{senzaLH}) is provided by any  {\it cohomology extension
 of the
fibre} $H^{*}(F)\stackrel{\theta }{\to} H^{*}(\tD)$ \cite[p. 256-258]{Spanier}.
Futhermore, the splitting (\ref{fibration}) can be viewed as a consequence of the Leray-Hirsch Theorem
(\cite[p. 182 and p. 195]{Voisin}, \cite[Lemma 2.5 and
proof]{DGF2}):
\begin{equation}\label{LHfirst}
 R\rho{_*}\bQ _{\widetilde{\Delta}}\stackrel{\theta }{\cong}\oplus_{\alpha=0}^{2\dim F} H^{\alpha}(F)\otimes \mathbb Q_{\Delta}[-\alpha],
\, \, \,\, \sideset{_{}^{p}}{_{}^{m}}{\mathop \hc}({\imath}{^*}R\pi{_*}\bQ _{\widetilde X})\cong H^{m-\dim \Delta}(F)\otimes \bQ _{ \Delta}[\dim \Delta].
\end{equation}
\noindent
The meaning of the isomorphism (\ref{LHfirst}) is that to each cohomology class $a \in H^{\alpha}(F)$ corresponds the morphism 
$\theta (a)\in H^{\alpha}(\tD)\cong Hom _{D^b_c(\Delta)}(\bQ _{\Delta}[-\alpha], R\rho{_*}\bQ _{\widetilde \Delta})$
so that 
$$ \theta \in Hom _{D^b_c(\Delta)}(\oplus_{\alpha=0}^{2\dim F} H^{\alpha}(F)\otimes \mathbb Q_{\Delta}[-\alpha], R\rho{_*}\bQ _{\widetilde \Delta}).$$
The splitting (\ref{LHfirst}) is NOT uniquely determined, rather it depends on the choice of a cohomology extension $H^{*}(F)\stackrel{\theta }{\to} H^{*}(\tD)$.
Specifically, any other choice 
$$\theta '= \theta + \omega \,\,\,\,\, \text{with}\,\,\,\,\,
Im(\omega )\subset Ker(H^{*}(\tD)\to H^{*}(F)) $$
 would give a different splitting in (\ref{LHfirst}). 
\end{remark}

Fix now two cohomology extensions
 of the
fibre $\theta, \, \overline{\theta}: H^{*}(F)\to H^{*}(\tD)$. Then, the morphism $\gamma$ above becomes
\begin{equation}
\label{2cohext}
\oplus_{\alpha=0}^{2\dim F} H^{\alpha}(F)\otimes \mathbb Q_{\Delta}[-\alpha]
\stackrel{\theta }{\cong} R{\rho}{_*}\mathbb  Q_{\tD}\to  R{\rho}{_*}\mathbb Q_{\tD}[2d]\stackrel{\overline{\theta} }{\cong}
\left( \oplus_{\alpha=0}^{2\dim F} H^{\alpha}(F)\otimes \mathbb Q_{\Delta}[-\alpha]\right)[2d].
\end{equation}
If we fix a basis $B$ of $H^{\bullet}(F)$ and a cohomology class $a\in B$, then the Leray-Hirsch Theorem implies that
\begin{equation}\label{tensorwithkappa}
\kappa \cup \theta(a)= \sum_{b\in B} \rho^*(\omega_{a \, b})\cup \overline{\theta }(b)\in H^{\deg a +2d}(\tD)\stackrel{\overline{\theta}}{\cong} \oplus _{i=0}^{2\dim F} (H^{\deg a+2d -i}(\Delta)\otimes H^{i}(F)),
\end{equation}
 for suitable and uniquely determined cohomology classes 
$\omega_{a \, b}\in H^{\deg a +2d - \deg b}(\Delta)\cong$ \par\noindent
$\cong Hom_{D_c^b(\Delta)}(\bQ_{\Delta}[-\deg a], \bQ_{\Delta}[ 2d - \deg b])$.
Then the cohomology class $\omega_{a \, b}$ represents the ``component''  of the morphism (\ref{2cohext}) with respect to the basis elements $a$, $b$ :
\begin{equation}
\label{comp2cohext}
\omega_{a\, b}: a\otimes \mathbb Q_{\Delta}[-\deg a] \to 
b \otimes \mathbb Q_{\Delta}[-\deg b + 2d].
\end{equation}
\par\noindent

\medskip

\begin{remark}
\label{differentcohext} Similarly as above, we have
\begin{equation}\label{differenttheta}
 \overline{\theta }(a)= \sum_{b\in B} \rho^*(\epsilon_{a \, b})\cup \theta(b)\in H^{\deg a }(\tD)\stackrel{\theta}{\cong} \oplus _{i} (H^{\deg a -i}(\Delta)\otimes H^{i}(F)),
\end{equation}
for  uniquely determined cohomology classes 
$$\epsilon_{a \, b}\in H^{\deg a  - \deg b}(\Delta)\cong 
 Hom_{D_c^b(\Delta)}(\bQ_{\Delta}[-\deg a], \bQ_{\Delta}[  - \deg b]).$$
The meaning of  (\ref{differenttheta}) is that to each cohomology class $\rho^*(\epsilon_{a \, b})\cup \theta(b)$ corresponds the compositum 
$$a\otimes \bQ_{\Delta}[- \deg a]\stackrel{\epsilon_{a \, b}}{\longrightarrow} b\otimes \bQ_{\Delta}[- \deg b]\stackrel{\theta(b)}{\longrightarrow}
 R\rho{_*}\bQ _{\widetilde \Delta}.$$
In other words, we have $\overline{\theta}=\theta \circ (\epsilon_{a \, b})$, with the matrix $(\epsilon_{a \, b})$ lying in
\par\noindent
$$Hom _{D^b_c(\Delta)}(\oplus_{\alpha=0}^{2\dim F} H^{\alpha}(F)\otimes \mathbb Q_{\Delta}[-\alpha], \oplus_{\alpha=0}^{2\dim F} H^{\alpha}(F)\otimes \mathbb Q_{\Delta}[-\alpha]).$$

Obviously we have $\epsilon_{a \, b}=0$, if $\deg b> \deg a$. Furthermore, by the very definition of cohomology
extension, the compositum of any   $H^{*}(F)\stackrel{\theta }{\to} H^{*}(\tD)$ with the projection on  $ H^{\deg a}(F)$ is the identity map.
Thus we have:
\begin{equation}
\label{belong}
\theta(a)- \overline{\theta}(a) \in \oplus _{i=0}^{\deg a -1} H^{i}(F)\otimes Hom_{D_c^b(\Delta)}(\bQ_{\Delta}[-\deg a], \bQ_{\Delta}[-i])
\cong Hom_{D_c^b(\Delta)}(\bQ_{\Delta}[-\deg a], (\sideset{_{}^{p}}{_{\leq k}^{}}{\mathop \tau} R\rho{_*}\bQ _{\widetilde{\Delta}})),
\end{equation}
\end{remark}
\noindent
where $k:= \deg a + \dim \Delta -1$. In particular,  the matrix $(\epsilon_{a \, b})$ is \textit{upper triangular}.
\medskip

\begin{remark}
\label{PieriI}
\begin{enumerate}
\item By Pieri's formula \cite[Formula 3.2.8]{Manivel},
the map 
$$D^{2\mid \lambda \mid}_{pq}\stackrel {\cup\,
c_{pq}^{c+1-q}}{\longrightarrow} A^{2(\mid \lambda \mid+ d_{pq})}_{pq},$$ 
determined by
cup-product with $c_{pq}^{c+1-q}$ (compare with Corollary \ref{cupc}), is injective  for all  $\mid  \lambda \mid \leq (p-q) \times(k-c+q-p)$.
In what follows, we denote by $E^{2(\mid \lambda \mid+ d_{pq})}_{pq}\subset A^{2(\mid \lambda \mid+ d_{pq})}_{pq}$ the image of the previous cup-product. Of course we have
\begin{equation}\label{Pieridim}
\cup\,c_{pq}^{c+1-q}: D^{2\mid \lambda \mid}_{pq} \leftrightarrow E^{2(\mid \lambda \mid+ d_{pq})}_{pq} \quad {\text  and } \quad   
dimD^{2\mid \lambda \mid}_{pq} = dim E^{2(\mid \lambda \mid+ d_{pq})}_{pq}
\end{equation}
Again by Pieri's formula, $E^{2(\mid \lambda \mid+ d_{pq})}_{pq}\subset A^{2(\mid \lambda \mid+ d_{pq})}_{pq}$
has a canonical basis which is in one-to-one correspondence with all partitions
$\lambda =(\lambda_1 , \dots ,\lambda _{p-q} )$ such that
$\lambda _{i}\geq (c+1-q)$, $1\leq i \leq p-q$.

\item Observe that partitions that are NOT contained in a 
$(p-q)\times (k-c+q-p)$ rectangle give a basis for the kernel of  the cup product 
$c_{pq}^{c+1-q}: H^{\bullet}(\bG_{p-q}(\bC ^{i_q}))\to E^{\bullet + d_{pq}}_{pq}$, and that 
{\it the composite } 
$$(c_{pq}^{c+1-q}){^{-1 }}\circ c_{pq}^{c+1-q}: H^{\bullet}(\bG_{p-q}(\bC ^{i_q}))\to E^{\bullet + d_{pq}}_{pq}
\to D^{\bullet}_{pq}$$ {\it coincides with the pull-back map}
$ H^{\bullet}(\bG_{p-q}(\bC ^{i_q}))\to H^{\bullet}(\bG_{p-q}(\bC ^{k-c}))\cong D^{\bullet}_{pq}$ 
(compare with Notations \ref{Dpq}).
\end{enumerate}
\end{remark}

Consider again the commutative diagram (\ref{D2}) and set $\jmath_{pq} :=\tilde\iota_{pq}\circ \tilde\pi_{pq}$.
As we said,   the map $\jmath_{pq}$   
induces   a Gysin morphism 
$R{\jmath_{pq}}{_*}\mathbb Q_{\tDpq}\to \mathbb Q_{\tD_p}[2d_{pq}]$ and, via composition with $\pi_p$, an arrow

\begin{equation}\label{trueGysin}
R{(\pi_p \circ \jmath_{pq})}{_*}\mathbb  Q_{\tDpq}=R{(\iota_{pq}\circ\pi_q \circ \rho_{pq})}{_*}\mathbb  Q_{\tDpq}\to R{\pi_p }{_*}\mathbb Q_{\tD_p}[2d_{pq}], \quad \text{in} \quad D_c^b(\Delta_p).
\end{equation}

By the Leray-Hirsch Theorem, we have
$
\bigoplus_{\alpha=0}^{2k_{pq}} A^{\alpha}_{pq}\otimes \mathbb Q_{\widetilde{\Delta}_q}[-\alpha]
\stackrel{\theta _{pq}}{\cong} R\rho_{pq*}\mathbb Q_{\widetilde\Delta_{pq}},$
for any cohomology extension $\theta _{pq}$,
and the above morphism become 

\begin{equation}\label{spqLH}
s_{pq}: R\iota_{pq}{_*}\left(\bigoplus_{\alpha=0}^{2k_{pq}} A^{\alpha}_{pq}\otimes  R{\pi}_{q \, *}\mathbb Q_{\widetilde{\Delta}_q}[-\alpha] \right)
 \stackrel{\theta _{pq}}{\cong} R(\iota_{pq}\circ \pi_q \circ \rho_{pq}){_*}\mathbb Q_{\tDpq} 
\to R{\pi_p}_{*}\mathbb Q_{\tD_p}[2d_{pq}].
\end{equation}

\noindent
Similarly,  applying   the Leray-Hirsch Theorem to the compositum of $\pi_p$ with the pull-back morphism we get

$$
\iota_{pq}{^*} R{\pi_p}_{*}\mathbb Q_{\tD_p}[2d_{pq}]\to
R( \pi_q \circ \rho_{pq})_*\mathbb Q_{\widetilde\Delta_{pq}}[2d_{pq}]\stackrel{\overline{\theta} _{pq}}{\cong}
\bigoplus_{\alpha=0}^{2k_{pq}} A^{\alpha+2d_{pq}}_{pq}\otimes 
R{\pi}_{q \, *} \mathbb Q_{\widetilde{\Delta}_q}[-\alpha]
,\,\, \text{in} \,\, D_c^b(\Delta_q),
$$
and
\begin{equation}
\label{full}
\bigoplus_{\alpha=0}^{2k_{pq}} A^{\alpha}_{pq}\otimes  R{\pi}_{q \, *}\mathbb Q_{\widetilde{\Delta}_q}[-\alpha] 
 \stackrel{\theta _{pq}}{\cong} R( \pi_q \circ \rho_{pq}){_*}\mathbb Q_{\tDpq} 
\to \iota_{pq}{^*}R{\pi_p}_{*}\mathbb Q_{\tD_p}[2d_{pq}]
\stackrel{\overline{\theta} _{pq}}{\cong}
\bigoplus_{\alpha=0}^{2k_{pq}} A^{\alpha+2d_{pq}}_{pq}\otimes 
R{\pi}_{q \, *} \mathbb Q_{\widetilde{\Delta}_q}[-\alpha]
\end{equation}
in $ D_c^b(\Delta_q)$ (compare with (\ref{2cohext})).
\medskip
\noindent
\medskip
\noindent On the other hand, the Decomposition Theorem tells us that $IC_{\Delta_q}^{\bullet}[-m_q]$ is  split embedded in 
$R{\pi}_{q \, *}\mathbb
Q_{\widetilde{\Delta}_q}$, thus we may
define
\begin{equation}\label{sigmapqLH}
\sigma_{pq}: R\iota_{pq}{_*}\left(\bigoplus_{\alpha=0}^{2k_{pq}} A^{\alpha}_{pq}\otimes  IC_{\Delta_q}^{\bullet}[-m_q-\alpha] \right) 
\to R{\pi_p}_{*}\mathbb Q_{\tD_p}[2d_{pq}],\,\, \text{in} \,\, D_c^b(\Delta_p),
\end{equation}

\begin{equation}
\label{taupqLH}
\tau_{pq}: \iota_{pq}{^*} R{\pi_p}_{*}\mathbb Q_{\tD_p}[2d_{pq}]\to
\bigoplus_{\alpha=0}^{2k_{pq}} A^{\alpha+2d_{pq}}_{pq}\otimes 
 IC_{\Delta_q}^{\bullet}[-m_q-\alpha]
,\,\, \text{in} \,\, D_c^b(\Delta_q),
\end{equation}

\begin{equation}
\label{gammapqLH}
\gamma_{pq}:=\tau_{pq}\circ \iota_{pq}{^*}(\sigma _{pq}): \bigoplus_{\alpha=0}^{2k_{pq}} A^{\alpha}_{pq}\otimes 
IC_{\Delta_q}^{\bullet}[-m_q-\alpha]\rightarrow
\bigoplus_{\alpha=0}^{2k_{pq}} A^{\alpha+2d_{pq}}_{pq}\otimes 
IC_{\Delta_q}^{\bullet}[-m_q-\alpha],\,\, \text{in} \,\, D_c^b(\Delta_q).
\end{equation}

\medskip

In the following Theorem recall that $\delta_{pq}=m_p-m_q -2d_{pq}$ (cf. \S 2.6).
\begin{theorem}
\label{isoDE} For every cohomology extension $\theta_{pq}$ and for every $i$
the compositum 
$$
 D_{pq}^{\delta_{pq}+i}\otimes_{\mathbb Q}
R\iota_{pq}{_*}IC_{\Delta_q}^{\bullet}[-i]
\stackrel{\sigma_{pq}}{\rightarrow}
R{\pi_p}_{*}\mathbb Q_{\tD_p}[m_p]
\rightarrow
\sideset{_{}^{p}}{_{}^{i}}{\mathop \hc}(R\,\pi_{p *}\mathbb Q_{\tD_{p}}[m_{p}])_{\Delta_q}[-i]
$$
is an isomorphism. Furthermore, the isomorphism above is independent of the cohomology extension.
\end{theorem}

\begin{proof}
We start the proof with the following
\vskip2mm
\noindent
\textit{Claim: for every}  $\theta_{pq}$ \textit{there exists} $\overline{\theta}_{pq}$ \textit{such that the map} $\gamma_{pq}$ \textit{induces an isomorphism}
$$
\gamma^0: D_{pq}^{\alpha}\otimes_{\mathbb Q}
\bQ _{\Delta_q^0}[-\alpha]\longleftrightarrow
E_{pq}^{\alpha+2d_{pq}}\otimes_{\mathbb Q}
\bQ _{\Delta_q^0}[-\alpha], \quad \forall \, 0\leq \alpha \leq \delta_{pq},$$
\textit{where} $\gamma^0:= \gamma_{pq}\mid_{\Delta^0_q}$ \textit{denotes the restriction to the smooth locus}.
\vskip2mm

Proceeding as above  (cf. (\ref{cupwith})  and (\ref{comp2cohext})), we fix a basis $B$ of $D_{pq}^{\bullet}$ and choose $a\in B$. The  morphism induced by $\gamma^0$
on the summand
 $a\otimes \mathbb Q_{\Delta_q^0}$ is  the cup product
with $\kappa \cup\theta_{pq}(a) $, where $\kappa$ is the top Chern class  of the normal bundle of $\Delta_{pq}^0\subset \tD_p$ (cf. (\ref{cupwith})). By Corollary \ref{cupc}, the restriction of $\kappa$ to the fiber $F_{pq}$ 
of the fibration $\Delta_{pq}^0\to \Delta_q^0$ is $c_{pq}^{c+1-q}$. Furthermore,
by Remark \ref{PieriI}, the cup product with $c:=c_{pq}^{c+1-q}$ provides an isomorphism
$$
\cup\,c: D^{\alpha}_{pq} \leftrightarrow E^{\alpha+ d_{pq}}_{pq}. 
$$
We thus can define $\overline{\theta}_{pq}$ in such a way that
\begin{equation}\label{defthetabar}
\overline{\theta}_{pq}:E^{\bullet}_{pq}\to H^{\bullet}(\Delta_{pq}^0),  \quad \overline{\theta}_{pq}(a\cup c):=\kappa \cup\theta_{pq}(a), \,\, \forall \, a\in B.
\end{equation}
With notations as in  (\ref{comp2cohext}), the component $\gamma_{a , \, a\cup c }^0: a\otimes \mathbb Q_{\Delta_q^0}\to a\cup c \otimes \mathbb Q_{\Delta_q^0}$ is an isomorphism and any other component vanishes.
In other words, 
if we look at the morphism $\oplus_{\alpha=0}^{2\delta_{pq}}D_{pq}^{\alpha}\otimes_{\mathbb Q}
\mathbb Q_{\Delta_q^0}[-\alpha]\to 
\oplus_{\alpha=0}^{2\delta_{pq}}E_{pq}^{\alpha+2d_{pq}}\otimes_{\mathbb Q}
\mathbb Q_{\Delta_q^0}[-\alpha]$ 
in matrix notations with respect to the bases $B$ and $c\cup B$, then it is represented by the identity matrix. The asserted Claim follows at once. 

The intersection cohomology complex $IC_{\Delta _q}^{\bullet} $ is a simple object in the abelian category Perv($\Delta _q$) \cite[Theorem 5.2.12]{Dimca2}, so the Claim just proved implies 
that the compositum
\begin{equation}\label{gammalim}
\gamma_{pq}: D_{pq}^{\alpha}\otimes_{\mathbb Q}
IC_{\Delta_q}^{\bullet}
\stackrel{ \sigma _{pq}}{\rightarrow }\iota_{pq}{^*} R{\pi_p}_{*}\mathbb Q_{\tD_p}[m_q+ \alpha +2d_{pq}]
\stackrel{ \tau_{pq}}{\rightarrow}
A^{\alpha+2d_{pq}}_{pq}\otimes 
IC_{\Delta_q}^{\bullet}
\end{equation}
is a monomorphism in the abelian category Perv($\Delta_q$). In view of Theorem \ref{PerverseCohomology} and taking into account the Decomposition Theorem,
in order to conclude it suffice
to observe that  (\ref{gammalim}) factors as
$$D_{pq}^{\delta_{pq}+i}\otimes_{\mathbb Q}
IC_{\Delta_q}^{\bullet}
\rightarrow \iota_{pq}{^*} R{\pi_p}_{*}\mathbb Q_{\tD_p}[m_p + i]
\rightarrow
\sideset{_{}^{p}}{_{}^{i}}{\mathop \hc}(R\,\pi_{p *}\mathbb Q_{\tD_{p}}[m_{p}])_{\Delta_q}
\rightarrow
A^{\delta_{pq}+i+2d_{pq}}_{pq}\otimes 
IC_{\Delta_q}^{\bullet},$$ where we set $\alpha = i + \delta_{pq}$.

As for the asserted independence of the extension class, we recall that $\sigma_{pq}$ is deduced from $s_{pq}$ from the fact that  $IC_{\Delta_q}^{\bullet}[-m_q]$ is  split embedded in 
$R{\pi}_{q \, *}\mathbb
Q_{\widetilde{\Delta}_q}$ (cf. (\ref{spqLH}) and (\ref{sigmapqLH})). Hence, it is enough to prove that the compositum

$$ D_{pq}^{\delta_{pq}+i}\otimes  R({\iota_{pq} \circ \pi_q})_{ *}\mathbb Q_{\widetilde{\Delta}_q}[m_q-i] 
 \stackrel{\sigma _{pq}}{\to} 
 R{\pi_p}_{*}\mathbb Q_{\tD_p}[m_p]
 \to \sideset{_{}^{p}}{_{}^{i}}{\mathop \hc}(R\,\pi_{p *}\mathbb Q_{\tD_{p}}[m_{p}])_{\Delta_q}[-i]
$$
 does not depend on the extension class. This is a consequence of Remark \ref{differentcohext} because the difference between
 two cohomology extensions factor as
 $$
 D_{pq}^{\delta_{pq}+i}\otimes  R({\iota_{pq} \circ \pi_q})_{ *}\mathbb Q_{\widetilde{\Delta}_q}[m_q-i] 
 \to
 \sideset{_{}^{p}}{_{\leq i-1}^{}}{\mathop \tau} R{\pi_p}_{*}\mathbb Q_{\tD_p}[m_p]
 \to 
 \sideset{_{}^{p}}{_{}^{i}}{\mathop \hc}(R\,\pi_{p *}\mathbb Q_{\tD_{p}}[m_{p}])_{\Delta_q}[-i]
 $$
 and 
 $$Hom_{D^b_c(\Delta _p)}(\sideset{_{}^{p}}{_{\leq i-1}^{}}{\mathop \tau} R{\pi_p}_{*}\mathbb Q_{\tD_p}[m_p],\sideset{_{}^{p}}{_{}^{i}}{\mathop \hc}(R\,\pi_{p *}\mathbb Q_{\tD_{p}}[m_{p}])_{\Delta_q}[-i])  =0,$$
 since $\sideset{_{}^{p}}{_{\leq i-1}^{}}{\mathop \tau} R{\pi_p}_{*}\mathbb Q_{\tD_p}[m_p] \in D_{\leq i-1}(\Delta _p)$ and
 $\sideset{_{}^{p}}{_{}^{i}}{\mathop \hc}(R\,\pi_{p *}\mathbb Q_{\tD_{p}}[m_{p}])_{\Delta_q}[-i]) \in D_{i}(\Delta _p)$.
\end{proof}

\medskip

\begin{theorem}\label{Gysinsplitting} For every choice of  cohomology extensions $\theta_{pq}$, $q<p$, we have:
\begin{enumerate}
\item 
the compositum 
$$
 \oplus_{q<p}D_{pq}^{\delta_{pq}+i}\otimes_{\mathbb Q}
R\iota_{pq}{_*}IC_{\Delta_q}^{\bullet}[-i]
\stackrel{\oplus \sigma_{pq}}{\longrightarrow}
R{\pi_p}_{*}\mathbb Q_{\tD_p}[m_p]
\rightarrow
\sideset{_{}^{p}}{_{}^{i}}{\mathop \hc}(R\,\pi_{p *}\mathbb Q_{\tD_{p}}[m_{p}])[-i]
$$
is a canonical isomorphism, $\forall \,i>0$;
\item the morphism $\sigma_{pq}$ gives a split embedding of the complex 
$\sideset{_{}^{p}}{_{}^{}}{\mathop \hc}(R\,\pi_{p *}\mathbb Q_{\tD_{p}}[m_{p}])_{\Delta_q}$
\par\noindent
$ 
\cong\bigoplus_{i=-\delta_{pq}}^{\delta_{pq}} D_{pq}^{\delta_{pq}+i}\otimes_{\mathbb Q}  R\iota_{pq}{_*}IC_{\Delta_q}^{\bullet}[-i] $ into the complex
 $ R{\pi_p}_{*}\mathbb Q_{\tD_p}[m_p]$;
\item the morphism $\oplus_{q<p}\sigma_{pq}$ gives a split embedding of the complex
$$\bigoplus_{q<p}\left(\sideset{_{}^{p}}{_{}^{}}{\mathop \hc}(R\,\pi_{p *}\mathbb Q_{\tD_{p}}[m_{p}])_{\Delta_q}\right)\cong
\bigoplus_{q<p}\left(\bigoplus_{i=-\delta_{pq}}^{i=\delta_{pq}}D_{pq}^{\delta_{pq}+i}\otimes_{\mathbb Q}R\iota_{pq}{_*}IC_{\Delta_q}^{\bullet}[-i]\right)$$
into the complex
 $ R{\pi_p}_{*}\mathbb Q_{\tD_p}[m_p].$
\end{enumerate}
Thus, a (non-canonical) splitting of $R{\pi_p}_{*}\mathbb Q_{\tD_p}[m_p]$ into simple objects is provided by the Gysin morphisms $\sigma_{pq}$, depending 
on the choice of the cohomology extensions $\theta_{pq}$, $\forall q<p$.
\end{theorem}
\begin{proof}
(1) The statement follows directly from  Theorem \ref{isoDE}, as each intersection cohomology complex $IC_{\Delta_q}^{\bullet}$ has neither sub-objects nor quotients 
supported on smaller strata \cite[\S 2.7]{DeCMReview}.

(2) It suffices to argue by induction, by applying Theorem \ref{isoDE} to the following commutative diagram with row split-exact sequences
$$
\begin{array}{ccccc}
 \tau_{\leq k-1}(\sideset{_{}^{p}}{_{}^{}}{\mathop \hc}(R\,\pi_{p *}\mathbb Q_{\tD_{p}}[m_{p}])_{\Delta_q})
 & \to  & \tau_{\leq k}(\sideset{_{}^{p}}{_{}^{}}{\mathop \hc}(R\,\pi_{p *}\mathbb Q_{\tD_{p}}[m_{p}])_{\Delta_q}) & \to 
 & \sideset{_{}^{p}}{_{}^{k}}{\mathop \hc}(R\,\pi_{p *}\mathbb Q_{\tD_{p}}[m_{p}])_{\Delta_q})[-k]  \\
  \updownarrow & & \uparrow  &  & \updownarrow \\
  \oplus_{i\leq k-1} D_{pq}^{\delta_{pq}+i}\otimes_{\mathbb Q}  R\iota_{pq}{_*}IC_{\Delta_q}^{\bullet}[-i] & \to 
 & \oplus_{i\leq k} D_{pq}^{\delta_{pq}+i}\otimes_{\mathbb Q} R\iota_{pq}{_*}IC_{\Delta_q}^{\bullet}[-i] & \to  & 
 D_{pq}^{\delta_{pq}+k}\otimes_{\mathbb Q}  R\iota_{pq}{_*}IC_{\Delta_q}^{\bullet}[-k].\\
\end{array}$$

(3) Similarly as above, the statement follows by induction by applying (1) to the following commutative diagram 
$$
\begin{array}{cccc}
 \tau_{\leq k-1}(\sideset{_{}^{p}}{_{}^{}}{\mathop \hc}(R\,\pi_{p *}\mathbb Q_{\tD_{p}}[m_{p}]))
 & \to  & \tau_{\leq k}(\sideset{_{}^{p}}{_{}^{}}{\mathop \hc}(R\,\pi_{p *}\mathbb Q_{\tD_{p}}[m_{p}])) & \to 
  \\
  \updownarrow & & \uparrow  &  \\
  \oplus_{q<p, \,i\leq k-1} D_{pq}^{\delta_{pq}+i}\otimes_{\mathbb Q}  R\iota_{pq}{_*}IC_{\Delta_q}^{\bullet}[-i] & \to 
 & \oplus_{q<p, \, i\leq k} D_{pq}^{\delta_{pq}+i}\otimes_{\mathbb Q}  R\iota_{pq}{_*}IC_{\Delta_q}^{\bullet}[-i] & \to  \\
\end{array}$$
$$
\begin{array}{cc}
\to  & \sideset{_{}^{p}}{_{}^{k}}{\mathop \hc}(R\,\pi_{p *}\mathbb Q_{\tD_{p}}[m_{p}]))[-k]  \\
   & \updownarrow \\
\to  & \oplus _{q<p}D_{pq}^{\delta_{pq}+k}\otimes_{\mathbb Q}  R\iota_{pq}{_*}IC_{\Delta_q}^{\bullet}[-k].\\
\end{array}$$
\end{proof}

\begin{remark}
\label{commento}
Observe that the isomorphism
$$
\gamma^0: D_{pq}^{\alpha}\otimes_{\mathbb Q}
\bQ _{\Delta_q^0}[-\alpha]\longleftrightarrow
E_{pq}^{\alpha+2d_{pq}}\otimes_{\mathbb Q}
\bQ _{\Delta_q^0}[-\alpha], \quad \forall \, 0\leq \alpha \leq \delta_{pq}$$
established   in the proof of Theorem \ref{isoDE}, shows that 
in the isomorphism (\ref{canonical})  the perverse cohomology sheaves
are generated via the Gysin morphism by
 the cohomology classes of the fibre  that are ``divisible'' by the top Chern class of the normal bundle  described in \S 4 (compare with (\ref{Pieridim}).
\end{remark}

\medskip

\section{The first Deligne splitting}

Having defined an explicit non-canonical splitting in \S 5, we are now aimed at defining a canonical one. Our main results are
collected in Theorem \ref{Deligne}. We are going to prove that the \textit{primitive perverse cohomology complexes} can be canonically
defined by means of the primitive cohomology of the sub-Grassmannians $T_{pq}$. Furthermore, we prove that the canonical morphism defined by Deligne in \cite[Proposition 2.4]{Deligne}
can be obtained as in \S 5 as soon as the cohomology extensions are defined appropriately.
So, the first Deligne splitting (\cite[2.5.1]{Deligne}) can be obtained by means of the Gysin maps defined in (\ref{sigmapqLH}), up to a suitable choice of the cohomology extensions.

\begin{notations}\label{matrix}
Let us consider again a fibration $\rho: \tD \to \Delta$ with fibre $F$ (cf. \S3 and \S5). Any cohomology class $\omega\in H^l(\tD )$ induces, via cup producy,
a morphim 
\begin{equation}
\label{cupomegaLH}
\mc _{\omega}:
\oplus_{\alpha=0}^{2dimF} H^{\alpha}(F)\otimes \mathbb Q_{\Delta}[-\alpha]
\stackrel{\theta }{\cong} R{\rho}{_*}\mathbb  Q_{\tD}\stackrel{\cup \omega }{\rightarrow}  
R{\rho}{_*}\mathbb Q_{\tD}[l]
\stackrel{\theta }{\cong}
\left( \oplus_{\alpha=0}^{2dimF} H^{\alpha}(F)\otimes \mathbb Q_{\Delta}[-\alpha]\right)[l],
\end{equation}
for any cohomology extension
 of the
fibre $\theta: H^{*}(F)\to H^{*}(\tD)$.
Similarly as above, 
if we fix a basis $B$ of $H^{\bullet}(F)$ and a cohomology class $a\in B$, then the Leray-Hirsch Theorem implies that
\begin{equation}\label{tensorwithomegaLH}
\omega \cup \theta(a)= \sum_{b\in B} \rho^*(\omega_{a \, b})\cup \theta (b)\in H^{\deg a +l}(\tD),
\end{equation}
 for well-determined cohomology classes 
$\omega_{a \, b}$, 
representing 
the entries of (\ref{cupomegaLH}): $\mc_{\omega}=(\omega_{a \, b})$, 
$\omega_{a \, b}\in Hom_{D^b_c(\Delta )}(a\otimes \mathbb Q_{\Delta}[-\deg a],
b \otimes \mathbb Q_{\Delta}[-\deg b + l])$.
\end{notations}

\begin{remark}\label{block}
Obviously,
$\omega_{a \, b}=0$ if $\deg a +l < \deg b $. Furthermore,
by definition of cohomology extension we have
$$
\sum_{b\in B\cap H^{l+ \deg a}} \rho^*(\omega_{a \, b})\cup \theta (b)=\theta(\sum_{b\in B\cap H^{l+ \deg a}} \omega_{a \, b}\, b) =\theta (\omega \mid _F \cup a),
$$
where $\omega \mid_F$ denotes the pull-back of $\omega$ on $F$ (observe that  $\omega_{a \, b} \in \bQ$ in the previous sum). 

In other words, 
 \textit{the  block of the matrix} $\mc _{\omega}$
 \textit{lying in} 
 $Hom_{D^b_c(\Delta )}(H^{\alpha}(F)\otimes \mathbb Q_{\Delta}[- \alpha], H^{\alpha +l}(F) \otimes \mathbb Q_{\Delta}[-\alpha ])$
    \textit{is determined by the extension class and depends only on the pull-back of $\omega$ to} $F$.
\end{remark}

\begin{lemma}
Consider the morphism $s_{pq}$ introduced in (\ref{spqLH}), fix a cohomology class $\omega\in H^{2l}(\tD_p)$ and denote by $\overline{\omega}\in H^{2l}(\tDpq)$ the pull-back of $\omega$ via $\jmath _{pq}$.
Then we have
$$ \pi_{p \, *}(\cup\omega )\circ s_{pq}= s_{pq}\circ ( \iota_{pq} \circ \pi_q )_* ( \mc_{\overline{\omega}}) $$
in
$Hom_{D^b_c(\Delta _p )}\left(R\iota_{pq}{_*}\left(\bigoplus_{\alpha=0}^{2k_{pq}} A^{\alpha}_{pq}\otimes  R{\pi}_{q \, *}\mathbb Q_{\widetilde{\Delta}_q}[-\alpha] \right), R{\pi_p}_{*}\mathbb Q_{\tD_p}[2d_{pq}+2l]\right).
$
\end{lemma}
\begin{proof}
The
 Gysin morphism 
$\Gamma_{pq}: R{\jmath_{pq}}{_*}\mathbb Q_{\tDpq}\to \mathbb Q_{\tD_p}[2d_{pq}]$ that we have considered in the previous section
is an example of \textit{Topological Bivariant class} \cite{FultonCF}, \cite{DeCM}:
$$\Gamma_{pq}\in H^{2d_{pq}}(\tDpq \stackrel{\jmath_{pq}}{\to} \tD_p) \cong Hom_{D^b_c(\tD_p)}^{\bullet}(R{\jmath_{pq}}{_*}\mathbb Q_{\tDpq}, \bQ_{\tD_p}).$$
Since the Topological Bivariant theory is \textit{skew-commutative} \cite[p. 22]{FultonCF},
we have 
$$\cup \omega \circ \Gamma_{pq}=\Gamma_{pq}\circ \cup \overline{\omega} \quad \text{in} \quad Hom_{D^b_c(\tD _p )}(R{\jmath_{pq}}{_*}\mathbb Q_{\tDpq}, \bQ_{\tD_p}[2l+2d_{pq}]),$$
for
any cohomology class $\omega \in H^{2l}(\tD_p)$.  By applying $\pi _p$, we get
$$\pi_{p \, *}(\cup\omega \circ   \Gamma_{pq})=\pi_{p \, *}(\Gamma_{pq}\circ \cup \overline{\omega} )
\quad \text{in} \quad Hom_{D^b_c(\Delta _p )}(R{(\pi_p \circ \jmath_{pq})}{_*}\mathbb Q_{\tDpq}, R{\pi _p}{_*}\bQ_{\tD _p}[2l+2d_{pq}]).$$
In view of  (\ref{spqLH}), $s_{pq}$ coincides with $\pi_{p \, *}(\Gamma_{pq}))$ applied on the complex
$$
R\iota_{pq}{_*}\left(\bigoplus_{\alpha=0}^{2k_{pq}} A^{\alpha}_{pq}\otimes  R{\pi}_{q \, *}\mathbb Q_{\widetilde{\Delta}_q}[-\alpha] \right)
 \stackrel{\theta _{pq}}{\cong} R(\iota_{pq}\circ \pi_q \circ \rho_{pq}){_*}\mathbb Q_{\tDpq} 
\cong R{(\pi_p \circ \jmath_{pq})}{_*}\mathbb Q_{\tDpq}.
$$
We conclude by taking into account the description of the matrix $ \mc_{\overline{\omega}}$ given in Notations \ref{matrix}.
\end{proof}
We recall that the map $\sigma_{pq}$ was defined from $s_{pq}$ via the embedding of $IC_{\Delta_q}^{\bullet}[-m_q]$ as a  split summand of
$R{\pi}_{q \, *}\mathbb
Q_{\widetilde{\Delta}_q}$ (cf. (\ref{spqLH}) and (\ref{sigmapqLH})). Thus we have the following

\begin{corollary}\label{afterlemma}
In
$Hom_{D^b_c(\Delta _p )}\left(R\iota_{pq}{_*}\left(\bigoplus_{\alpha=0}^{2k_{pq}} A^{\alpha}_{pq}\otimes IC^{\bullet}_{\widetilde{\Delta}_q}[-m_q-\alpha] \right), R{\pi_p}_{*}\mathbb Q_{\tD_p}[2d_{pq}+2l]\right)
$
we have
$$ \pi_{p \, *}(\cup\omega )\circ \sigma_{pq}= \sigma_{pq}\circ ( \iota_{pq} \circ \pi_q )_* ( \mc_{\overline{\omega}}). $$
\end{corollary} 

\medskip

Let $h$ be the first Chern class of a $\pi_p$-ample line bundle on $\tD_p$.
In the fundamental paper \cite{Deligne}, Deligne introduced a canonical splitting of $R\,\pi_{p *}\mathbb Q_{\tD_{p}}[m_{p}]$
 by means of the \textit{primitive perverse cohomology complexes}:
\begin{equation}
\label{PrimPerv}
\pc^{-i}(R\,\pi_{p *}\mathbb Q_{\tD_{p}}[m_{p}]):= Ker(\cup h^{i+1}: \sideset{_{}^{p}}{_{}^{-i}}{\mathop \hc}(R\,\pi_{p *}\mathbb Q_{\tD_{p}}[m_{p}])\longrightarrow
\sideset{_{}^{p}}{_{}^{i+2}}{\mathop \hc}(R\,\pi_{p *}\mathbb Q_{\tD_{p}}[m_{p}]).
\end{equation}

The following result shows how one can obtain a simple description of the primitive perverse cohomology complexes
and of Deligne's splitting in terms of the morphisms $\sigma_{pq}$ described in (\ref{sigmapqLH}):

\begin{theorem}
\label{Deligne} Let $h$ be the first Chern class of a $\pi_p$-ample line bundle on $\tD_p$ and let $h\mid_{T_{pq}}$ be the pull-back of $h$ to $T_{pq}$.
\begin{enumerate}
\item We have the following isomorphism
\begin{equation}
\label{primitive}
\pc^{-i}(R\,\pi_{p *}\mathbb Q_{\tD_{p}}[m_{p}])_{\Delta_q}\cong 
P^{\delta_{pq}-i}_{pq}\otimes_{\mathbb Q}R{\iota_{pq}}_* IC_{\Delta_q}^{\bullet} \, 
\lhd \sideset{_{}^{p}}{_{}^{-i}}{\mathop \hc}(R\,\pi_{p *}\mathbb Q_{\tD_{p}}[m_{p}])_{\Delta_q},
\end{equation}
where  $P^{\delta_{pq}-i}_{pq}:=Ker((h\mid_{T_{pq}})^{i+1}:D_{pq}^{\delta_{pq}-i}\to D_{pq}^{\delta_{pq}+i+2})$
denotes the primitive cohomology of the sub-grassmannian $T_{pq}$. Furthermore, the isomorphim (\ref{primitive}) is canonical, i.e. independent of the cohomology extension.
\item There exists a cohomology extension for which the restriction of  (\ref{sigmapqLH})
to the primitive perverse cohomology complexes:
$$\sigma_{pq}: \pc^{-i}(R\,\pi_{p *}\mathbb Q_{\tD_{p}}[m_{p}]) \longrightarrow R\,\pi_{p *}\mathbb Q_{\tD_{p}}[m_{p}]$$
coincides with the canonical morphism defined in \cite[Proposition 2.4]{Deligne}. 
\end{enumerate}
\end{theorem}

\begin{proof}
Point (1) is a direct consequence of Theorem \ref{Gysinsplitting}, Remark
\ref{block} and Corollary \ref{afterlemma}. Indeed, by  Theorem \ref{Gysinsplitting} we have canonical isomorphisms
$$
\sideset{_{}^{p}}{_{}^{-i}}{\mathop \hc}(R\,\pi_{p *}\mathbb Q_{\tD_{p}}[m_{p}])\stackrel{\sigma_{pq}}{\longleftrightarrow}
 D_{pq}^{\delta _{pq}-i}\otimes_{\mathbb Q}IC_{\Delta_q}^{\bullet}  \quad \text{and} \quad 
\sideset{_{}^{p}}{_{}^{i+2}}{\mathop \hc}(R\,\pi_{p *}\mathbb Q_{\tD_{p}}[m_{p}])
\stackrel{\sigma_{pq}}{\longleftrightarrow} D_{pq}^{\delta _{pq}+i+2}\otimes_{\mathbb Q}
IC_{\Delta_q}^{\bullet},
$$
hence the morphism $\cup h^{i+1}: \sideset{_{}^{p}}{_{}^{-i}}{\mathop \hc}(R\,\pi_{p *}\mathbb Q_{\tD_{p}}[m_{p}])\longrightarrow
\sideset{_{}^{p}}{_{}^{i+2}}{\mathop \hc}(R\,\pi_{p *}\mathbb Q_{\tD_{p}}[m_{p}])$ factors through 
$R\sigma_{pq}{_*}(D_{pq}^{\delta _{pq}+i+2}\otimes_{\mathbb Q} IC_{\Delta_q}^{\bullet})$. 
On the other hand, from Corollary \ref{afterlemma} we get
$$Ker(\cup h^{i+1}: \sideset{_{}^{p}}{_{}^{-i}}{\mathop \hc}(R\,\pi_{p *}\mathbb Q_{\tD_{p}}[m_{p}])\longrightarrow
\sideset{_{}^{p}}{_{}^{i+2}}{\mathop \hc}(R\,\pi_{p *}\mathbb Q_{\tD_{p}}[m_{p}])\cong$$
$$\cong Ker (\mc_{\overline{h}^{i+1}}: 
D_{pq}^{\delta _{pq}-i}\otimes_{\mathbb Q}IC_{\Delta_q}^{\bullet}\to
D_{pq}^{\delta _{pq}+i+2}\otimes_{\mathbb Q} IC_{\Delta_q}^{\bullet}).
$$
Finally, in view of Remark
\ref{block} we have 
$$ Ker (\mc_{\overline{h}^{i+1}}: 
D_{pq}^{\delta _{pq}-i}\otimes_{\mathbb Q}IC_{\Delta_q}^{\bullet}\to
D_{pq}^{\delta _{pq}+i+2}\otimes_{\mathbb Q} IC_{\Delta_q}^{\bullet}) \cong Ker((h\mid_{T_{pq}})^{i+1}\otimes_{\mathbb Q} IC_{\Delta_q}^{\bullet})
$$
and we are done. The independence of the cohomology extension follows directly from Theorem \ref{Gysinsplitting}.

As for  (2), by \cite[Proposition 2.4]{Deligne} and \cite[Lemma 2.4.1]{DeC} we need to prove that we are allowed to change  the cohomology extension in such a way that the compositum
\begin{equation}\label{todelete}
\pc^{-i}(R\,\pi_{p *}\mathbb Q_{\tD_{p}}[m_{p}])_{\Delta_q}[i]
\stackrel{h^{s}\circ\sigma_{pq}}{\longrightarrow} 
R\,\pi_{p *}\mathbb Q_{\tD_{p}}[m_{p}+2s]\to
\end{equation}
$$
\to 
\left(\tau _{\geq s}R\,\pi_{p *}\mathbb Q_{\tD_{p}}[m_{p}]_{\Delta_q}\right)[2s]
\cong
\bigoplus_{k\leq s-i}D_{pq}^{\delta_{pq}+i+k}\otimes_{\mathbb Q}
IC_{\Delta_q}^{\bullet}[i+k]
$$
vanishes when $s>i$. 
We argue by induction on $s$, starting from $s=i+1$.
By point (1) just proved it suffices to show that we can modify $\theta_{pq}$ so that
$$ P^{\delta_{pq}-i}_{pq}\otimes_{\mathbb Q}R{\iota_{pq}}_* IC_{\Delta_q}^{\bullet}
\stackrel{h^{i+1}\circ\sigma_{pq}}{\longrightarrow} 
 D_{pq}^{\delta_{pq}+i+1}\otimes_{\mathbb Q} IC_{\Delta_q}^{\bullet}[1]
$$
vanishes. With notations as in (\ref{cupomegaLH}) and Remark \ref{block}, it is defined by matrix elements of the form
$\sum_{b} \rho_{pq}^*(\omega_{a \, b})\cup \theta_{pq} (b)
$, with $b$ varying into $ B\cap H^{\delta_{pq}+i+1}(T_{pq})$.
But  such matrix elements must vanish since $\omega_{a \, b}\in H^1(\tD_q)=0$.

In general, by induction it is enough to show that we can modify $\theta_{pq}$ so that
$$\left(P^{\delta_{pq}-i}_{pq}\otimes_{\mathbb Q}R{\iota_{pq}}_* IC_{\Delta_q}^{\bullet}
\stackrel{h^{i+k}\circ\sigma_{pq}}{\longrightarrow} 
 D_{pq}^{\delta_{pq}+i+k}\otimes_{\mathbb Q} IC_{\Delta_q}^{\bullet}[k]\right)=0.
$$
By Corollay \ref{afterlemma}, the cup product with $h^{i+k}$ is controlled by the matrix $\mc_{\overline {h}^{i+k}}$ defined in   \ref{matrix}:
\begin{equation}\label{hs}
\overline{h}^{i+k} \cup \theta_{pq}(a)= \sum_{b\in B} \rho^*(\omega_{a \, b})\cup \theta_{pq} (b).
\end{equation}
Again, by Remark \ref{block} the portion of the previous sum concerning the block we are interested in is
$$\sum_{b} \rho_{pq}^*(\omega_{a \, b})\cup \theta_{pq} (b), \quad \omega_{a \, b}\in H^k(\tD_q)\subset Hom(IC_{\Delta_q}^{\bullet}, IC_{\Delta_q}^{\bullet}[k]), \quad b\in B\cap H^{\delta_{pq}+i+k}(T_{pq}).$$
By the hard Lefschetz theorem for $T_{pq}$, we have $b=(h\mid_{T_{pq}})^{i+k}\cup b^{-1}$
for uniquely determined cohomology classes $b^{-1}\in H^{\delta_{pq}-i-k}(T_{pq})$. 

Furthermore, by Remark \ref{differentcohext}  
$
\rho_{pq}^*(\omega_{a \, b})\cup\theta ((h\mid_{T_{pq}})^{i+k}\cup b^{-1})$
and $\rho_{pq}^*(\omega_{a \, b})\cup \overline{h}^{i+k}\cup \theta(b^{-1}) 
$ have the same projection in $Hom_{D^b_c(\Delta_q)}(IC_{\Delta_q}^{\bullet}, IC_{\Delta_q}^{\bullet}[k])$.
In other words, the difference 
$$\overline{h}^{i+k} \cup \theta_{pq}(a)- \sum_{b} \rho_{pq}^*(\omega_{a \, b})\cup \overline{h}^{i+k}\cup \theta_{pq}(b^{-1}), \quad  b\in B\cap H^{\delta_{pq}+i+k}(T_{pq})$$
has vanishing matrix elements in $Hom_{D^b_c(\Delta_q)}(IC_{\Delta_q}^{\bullet}, IC_{\Delta_q}^{\bullet}[l])$, whenever $l\leq k$.
So such a difference does not affect (\ref{todelete}). For our purposes it suffices to modify the cohomology extension according to
$$\theta_{pq}(a)':=\theta_{pq}(a)-
\sum_{b} \rho_{pq}^*(\omega_{a \, b})\cup \theta_{pq}(b^{-1}),\quad a\in  P^{\delta_{pq}-i}_{pq},\quad  b\in B\cap H^{\delta_{pq}+i+k}(T_{pq}).
$$
\end{proof}
 
\begin{remark}
\label{Delignesplitting} Theorem \ref{Deligne} implies that the first Deligne splitting (\cite[2.5.1]{Deligne}) can be obtained by means of the Gysin maps defined in (\ref{sigmapqLH}), for a sutable choice of the cohomology extensions.
\end{remark}

\medskip

\section{An $h$-good splitting.}

\begin{notations}\label{h}
Let us denote by $\oc_{\bG_{i_r}(F)}(1)$ the hyperplane divisor of the Pl\"ucker embedding of $\bG_{i_r}(F)$. Since any fiber of $\pi_r$ is a subgrassmannian of $\bG_{i_r}(F)$, the 
pull-back  
$f_r^*(\oc_{\bG_{i_r}(F)} (1))$   is $\pi_r$-ample on $\tD_r$. In what follows, we denote by $h$ the first Chern class of the pull-back  $f_r^*(\oc_{\bG_{i_r}(F)} (1))$.
\end{notations}

Let $H:= H^{\bullet}(\tD_r)$ be the total cohomology vector space of $\tD_r$. The projection $\pi_r$ endows $H$ with the \textit{perverse Leray  filtration} 
$$ H_p :=P_p H \slash P_{p-1} H\cong \bH^{\bullet}(\sideset{_{}^{p}}{_{}^{p}}{\mathop \hc}(R\,\pi_{r *}\mathbb Q_{\tD_{r}}[m_{r}])). 
$$
In the paper \cite{DeC}, de Cataldo considered  two objects $\bH:= (H, P)$ and $\bH_*:= \oplus_p (H_p, T[-p])$ ($T[-p]$ denotes the trivial filtration shifted to position $p$), in the filtered category of rational vector spaces
and 
constructed five distingushed mixed-Hodge theoretic good splittings $\bH \cong \bH_*$ \cite[Theorem 1.1.1]{DeC}.

In general, the five good splittings turn out to be pairwise distinct but there is a natural condition, the existence of an $e$-good splitting
\cite[Definition 2.4.5]{DeC}, under which they in fact coincide \cite[Theorem 2.6.3]{DeC}. Such a condition says that the cup product with $e$, the 
first Chern class of a relatively ample line bundle, is \textit{homogeneous of degree two} in $\bH_*$.

Theorem  \ref{hgood} proves that this is indeed the case for the cup product with the first Chern class of the pull-back  $f_r^*(\oc_{\bG_{i_r}(F)} (1))$.

\begin{notations}\label{Apq}
Recall from Remark \ref{NewRemark} (3) that,   for all $\beta\in\mathbb Z$,  we set
$A^{\beta}_{pq}:=H^{\beta}(F_{pq})=H^{\beta}(\bG_{i_p}(\bC ^{i_q}))$. 
By \cite[Corollary 3.2.4]{Manivel}, $A^{\beta}_{pq}=0$ for $\beta $ odd and $A^{2\mid \lambda \mid}_{pq}$ has a canonical basis which is in one-to-one correspondence with all partitions $\lambda$ contained in a 
$(p-q)\times i_p$ rectangle. Following \cite{Manivel},  for any partition $\lambda$, we  denote by $\sigma_{\lambda}\in A^{2\mid \lambda \mid}_{pq}$ the corresponding cohomology class.

We  define the following cohomology extension of the fibre:
\begin{equation}
\label{cohext}H^{*}(F_{pq})\stackrel{\theta _{pq}}{\to} H^{*}(\tDpq),\quad 
\theta _{pq}(\sigma_{\lambda}):=    f_p^*(s_{\lambda}(x_1, \dots , x_{i_p})), \quad \forall \,\sigma_{\lambda}\in A^{2\mid \lambda \mid}_{pq}
\end{equation}
 where $s_{\lambda}$ is the {\it Schur polynomial} corresponding to $\lambda$ and
$x_1, \dots , x_{i_p}$ are the {\it Chern roots} of the bundle $ S_{p}$ 
(compare with \cite[(2.6), p. 18]{FultonPr}). Of course there are infinitely many other choices for the cohomology extension, one of these will be considered in \S 9.
\end{notations}

\begin{theorem}
\label{hgood} Assume that the cohomology extensions are  defined as in (\ref{cohext}). Then
the splitting $\bH \cong \bH_*$ is $h$-good. Specifically, the cup product with $h$ induces on each
$\sideset{_{}^{p}}{_{}^{}}{\mathop \hc}(R\,\pi_{r *}\mathbb Q_{\tD_{p}}[m_{r}])_{\Delta_q}$ a morphism which is homogeneous of degree two.
\end{theorem}
\begin{proof}
It suffices to prove the second statement since, by Theorem \ref{Gysinsplitting}, the $i$-th graded piece of the splitting is the image of the map induced in hypercohomology by
the component $\oplus_{q<r}D_{rq }^{\delta_{rq} + i}\otimes_{\mathbb Q}
R{\iota_{rq}}_* IC_{\Delta_q}^{\bullet}\to R\,\pi_{r *}\mathbb Q_{\tD_r}[m_r]$
of $\oplus_{q<r}\sigma_{rq}$.

The proof is very similar to the one of Theorem \ref{Deligne}.
We have to prove that the compositum
\begin{equation}\label{cuph}
D_{rq}^{\delta_{rq}+i}\otimes_{\mathbb Q}
R\iota_{rq}IC_{\Delta_q}^{\bullet}[-i]
\stackrel{\sigma_{rq}}{\rightarrow}
R{\pi_r}_{*}\mathbb Q_{\tD_r}[m_r]
\stackrel{\cup h}{\rightarrow}
R{\pi_r}_{*}\mathbb Q_{\tD_r}[m_r+2]
\rightarrow 
D_{rq}^{\delta_{rq}+2+j}\otimes_{\mathbb Q}
R\iota_{pr}IC_{\Delta_q}^{\bullet}[-j]
\end{equation}
vanishes whenever $i\not =j$. By Theorem \ref{Gysinsplitting} and Corollary \ref{afterlemma}, this amounts to  prove the vanishing of any
 matrix element of $\mc _{\overline{h}}$ lying in  
$$Hom_{D^b_c(\tD_q )}(a\otimes \mathbb Q_{\tD_q}[-i], b \otimes \mathbb Q_{\tD_q}[-j]), \quad a\in D_{rq}^{\delta_{rq}+i},     \, \,    b\in D_{rq}^{\delta_{rq}+2+j}$$
(cf. Notations \ref{matrix}).

The matrix elements of  $\mc _{\overline{h}}$ are defined by
\begin{equation}\label{tensorwithkappah}
\overline{h} \cup \theta_{rq}(a)= \sum\rho_{rq}^*(c^a_b)\cup \theta _{rq}(b),
\end{equation}
 for well determined cohomology classes   
$c_b^a\in H^{i +2-j}(\tD_q)$
(cf. (\ref{matrix})).
Hence, in order to prove (\ref{cuph}), it suffices to show that 
\begin{equation}\label{cbzero}
c_b^a=0,  \quad \text{for all} \quad a\in D_{rq}^{\delta_{pq}+i},     \,    \, b\in D_{rq}^{\delta_{pq}+2+j}, \quad \text{whenever} \quad i\not = j.
\end{equation}
\noindent
If we denote by $\lambda _{\omega}$ the partition associated to $\omega\in D_{rq}^{\bullet}$ (cf. Notations \ref{Apq}),
by the definitions of  $\theta_{rq}$ and  $h$ (cf. (\ref{cohext})  and \ref{h}), we have
$$\overline{h} \cup \theta_{rq}(a)=
f_r^*(s_{\lambda_{h}}(x_1, \dots , x_{i_r}))\cup f_r^*(s_{\lambda_a}(x_1, \dots , x_{i_r}))=
f_r^*((s_{\lambda_h}\cup s_{\lambda_a}))=
\theta_{rq}(a\cup h\mid _{F_{rq}}),
$$
 where $s_{\bullet}$ are the  Schur polynomials  and
$x_1, \dots , x_{i_r}$  the Chern roots of the bundle $ S_r$.
Taking into account (\ref{tensorwithkappah}), (\ref{cbzero}) follows immediately.
\end{proof}

\begin{corollary}[de Cataldo]
Assume that the cohomology extensions are  defined as in (\ref{cohext}) and that $h$ 
is the first Chern class of   $f_r^*(\oc_{\bG_{i_r}(F)} (1))$. Then
the five splittings considered in \cite{DeC} coincide.
\end{corollary}

\medskip

\section{Local analysis of the intersection cohomology complex.}

In this section we are aimed at a more careful local analysis of the intersection cohomology complex $IC_{\Delta_p}$. We will prove in Theorem \ref{LocStudy}  that the splitting
above
leads to a canonical decomposition of the cohomoloy spaces of $F_{pq}$, where it is possible to ``recognize'' the summand coming from $B_{pq}$ (cf. ((\ref{gammaAlow}), Theorem \ref{LocStudy}
and Remark \ref{recognize1}).
For this purpose,
we find more convenient to change our previous choice of the cohomology extension (compare with (\ref{cohext})):

\begin{notations}
\label{newcohexdef} We denote by $\overline{\theta} _{pq}$ the following cohomology extension of the fibre
\begin{equation}\label{newcohex}
H^{*}(F_{pq})\stackrel{\overline{\theta} _{pq}}{\to} H^{*}(\tDpq),\quad 
\overline{\theta} _{pq}(\sigma_{\lambda}):=    s_{\lambda}(x_1, \dots , x_{p-q}), \quad \forall \,\sigma_{\lambda}\in A^{2\mid \lambda \mid}_{pq},
\end{equation}
where $s_{\lambda}$ is the  Schur polynomial corresponding to $\lambda$ and
$x_1, \dots , x_{p-q}$ are the  Chern roots of the bundle $S_{q}\slash S_{p}$ on $\tDpq$.
\end{notations}

 Recall from \S 2.4 and \S 2.5 that
$\tD_p=\bG_{p-1}(\kc_p)$ and
$\tDpq=\bG_{i_p}(f^*_q(S_q))$. So we have $\tDpq\subset \fc_{pq}:=\bF_{p-q, p-1}(\kc_p)$, the bundle of partial $(p-q, p-1)$-flags of the vector bundle $\kc_p$ on $\bG_{i_p}(F)$.
This implies that 
$$
{\text{\it the map $\jmath_{pq} :\tDpq \to \tD_p$ factorizes by the inclusion }}
$$
\begin{equation}
\label{inclusionDeltapq}
\zeta_{pq}: \tDpq\subset \fc_{pq}=\bF_{p-q, p-1}(\kc_p)
\end{equation}
$$
{\text{\it with the natural projection }}
$$
\begin{equation}
\label{projectionFpq}
\xi_{pq}:\fc_{pq}=\bF_{p-q, p-1}(\kc_p)\to \bG_{p-1}(\kc_p)= \tD_p.
\end{equation}

\noindent
\begin{remark}
\label{GysinFactorization}
\begin{enumerate} 
\item Since  a flag $(W,Z,V)\in \fc_{pq}$ belongs to $\tDpq$ iff $Z\subset F$, the map $\zeta_{pq}$ is a {\it regular imbedding }(\cite[p. 437]{FultonIT}) of codimension $c(p-q)$. So there
is an orientation class $\zeta _{pq\, !}\in \text{Hom}(\zeta_{pq}{_*} \mathbb Q_{\tDpq}, \mathbb Q_{\fc_{pq}}[2c(p-q)])$. Furthermore, by the self-intersection formula, 
$$
{\text{\it the composite of $\tau _{pq}$ with the pull-back morphism $\zeta_{pq}{^*}:\mathbb Q_{\fc_{pq}}\to \zeta_{pq}{_*} \mathbb Q_{\tDpq}$  }}
$$
$$
{\text{\it coincides with the  cup product with $c_{p-q}(S_{q}\slash S_{p})^c$:}}  
$$
\begin{equation}
\label{cupwithctop}
\zeta_{pq}{^*}\circ \zeta _{pq\, !}= \cup c_{p-q}(S_{q}\slash S_{p})^c: \mathbb Q_{\tDpq} \to \mathbb Q_{\tDpq}[2c(p-q)].
\end{equation}
\item The projection (\ref{projectionFpq}) gives rise to a  Gysin map 
$$\xi _{pq\, !}\in \text{Hom}(\xi _{pq}{_*} \mathbb Q_{\fc_{pq}}, \mathbb Q_{\tD_{p}}[-2(p-q)(q-1)])$$
 which  is a  particular case of the one described in \cite[Exercise 3.8.3]{Manivel}.
\item The Gysin morphism 
$R{\jmath_{pq}}{_*}\mathbb Q_{\tDpq}\to \mathbb Q_{\tD_p}[2d_{pq}]$ is the composition 
$\xi _{pq\, !}\circ \zeta _{pq\, !}$.
\end{enumerate}
\end{remark}

\medskip

Fix $l<q<p$ and consider the following commutative diagram
$$
\begin{array}{ccccccc}
 \tDpq ^l & \subset  & \widetilde{\Delta}_{pq} & \stackrel{\tilde\pi_{pq}}{\longrightarrow} & \Delta_{pq} & \stackrel{\tilde\iota_{pq}}{\hookrightarrow} & \widetilde\Delta_p \\
& & \stackrel {\rho_{pq}}{}\downarrow & &\stackrel {\pi_{pq}}{}\downarrow & & \stackrel {\pi_{p}}{}\downarrow  \\
& & \widetilde\Delta_q &\stackrel{\pi_q}{\longrightarrow} & \Delta_q &\stackrel{\iota_{pq}}{\hookrightarrow} & \Delta_p  \\
& &  \cup & & \cup & & \\
\tD_{ql}^0 & \stackrel{\tilde\pi_{ql}}{\cong} &  \Delta_{ql}^0 & \stackrel{\pi_{ql}}{\to}  & \Delta_l^0 & & \\
\end{array}
$$
where all the characters were previously defined except for $$\tDpq ^l:= \rho_{pq}^{-1}(\Delta_{ql}^0)\cong\bF_{i_p , i_q }(f^*_l(S_l)\mid _{\widetilde\Delta_l^0}).$$
The pull-back of (\ref{trueGysin}) to $\Delta_r^0$  gives rise to  a morphism in $D_c^b(\Delta_r^0)$:
\begin{equation}
\label{GysinDeltarpq}
 R{(\pi_{ql} \circ \rho_{pq})}{_*}\mathbb  Q_{\tDpq^l} \cong R{(\pi_p \circ \jmath_{pq})}{_*}\mathbb  Q_{\tDpq}\mid_{\tD_l^0} \to R{\pi_p }{_*}\mathbb Q_{\tD_p}[2d_{pq}]\mid_{\tD_l^0}
\cong R{\pi_{pq} }{_*}\mathbb Q_{\tD_{pr}^0}[2d_{pq}].
\end{equation}

\medskip

\begin{remark}
\label{ChernRootsFlag} Obviously, we have an inclusion $\tDpq \subset \bF_{i_p, i_q, k}(\bC ^l)$, so the Chern roots $(x_1, \dots , x_{p-q})$ introduced in Notations \ref{newcohexdef} can be seen as the pull-back in $\tDpq$ of  a subset of the roots
$(t_1, \dots , t_{t})$ of $\bF(\bC^l)$ \cite[p. 161]{FultonIT}, the variety of complete flags of $\bC^l$. For instance, in (\ref{newcohex}) we can also put 
$\overline{\theta} _{pq}(\sigma_{\lambda})=s_{\lambda}(x_1, \dots , x_{p-q})=s_{\lambda}(t_{i_p+1}, \dots , t_{i_q})$.
\end{remark}

\begin{notations}
Since the fibres of the projection:
$$\pi_{ql} \circ \rho_{pq}:\tDpq ^l\cong\bF_{i_p , i_q }(f^*_l(S_l)\mid _{\widetilde\Delta_l^0}) \to \widetilde\Delta_l^0$$
are Flag manifolds $\bF_{i_p , i_q }(\bC^{i_l})$  and $H^*(\bF_{i_p , i_q })=H^*(\bG_{i_p}(\bC^{i_q}) \times \bG_{i_q }(\bC^{i_l}))$,
we have also  a decomposition in $D^b_c(\Delta_l^0)$:
\begin{equation}\label{LHProduct}
\sum_{\beta=0}^{2k_{ql}}\sum_{\alpha=0}^{2k_{pq}} A^{\beta}_{ql}\otimes A^{\alpha}_{pq}\otimes \mathbb Q_{\Delta_l^0}[-\alpha- \beta]
\stackrel{\overline{\theta}_{pq}\otimes\overline{\theta}_{ql}}{\longleftrightarrow} R(\pi _{ql}\circ \rho_{pq}){_*} \mathbb Q_{\tDpq ^l}.
\end{equation}
\end{notations}

\medskip
Combining (\ref{GysinDeltarpq}) and (\ref{LHProduct}), we get a morphism in $D^b_c(\Delta _l^0)$:
$$
\gamma^l_{pq}: \sum_{\beta=0}^{2k_{ql}}\sum_{\alpha=0}^{2k_{pq}} A^{\beta}_{ql}\otimes A^{\alpha}_{pq}\otimes \mathbb Q_{\Delta_l^0}[-\alpha- \beta]
\stackrel{\overline{\theta}_{pq}\otimes\overline{\theta}_{ql}}{\longleftrightarrow}
 R{(\pi_{ql} \circ \rho_{pq})}{_*}\mathbb  Q_{\tDpq^l}  \to 
 $$
 $$\to R{\pi_{pl} }{_*}\mathbb Q_{\tD_{pl}^0}[2d_{pq}]\stackrel{\overline{\theta}_{pl}}{\cong}
\left( \sum_{\alpha=0}^{2k_{pl}} A^{\alpha}_{pl}\otimes R\pi_{pl*}\mathbb Q_{\Delta_l^0}[-\alpha] \right) [2d_{pq}]. 
$$

\begin{proposition}\label{gamma}
Fix $\sigma_{\lambda}\in A^{\beta}_{ql}$ and $\sigma_{\mu}\in A^{\alpha}_{pq}$, with $\lambda:=(\lambda_1, \dots , \lambda_{q-l})$ and $\mu:=(\mu_1, \dots , \mu_{p-q})$. 
If we set $\nu:=(\nu_1, \dots , \nu_{p-l})=(\mu_1+c+1-q, \dots , \mu_{p-q}+c+1-q, \lambda_1 , \dots , \lambda_{q-l})$ then we have
$$\gamma^l_{pq}=\overline{\theta}_{pl}(\sigma _{\nu}) \quad \text{in}$$
$$
Hom_{D^b_c(\Delta^0_{l})}(\sigma_{\lambda}\otimes \sigma_{\mu} \otimes \mathbb Q_{\Delta_l^0}[-\alpha- \beta], R{\pi_{pl} }{_*}\mathbb Q_{\tD_{pl}^0}[2d_{pq}])
\cong Hom_{D^b_c(\Delta^0_{l})}( \mathbb Q_{\Delta_l^0}[-\alpha- \beta], R{\pi_{pl} }{_*}\mathbb Q_{\tD_{pl}^0}[2d_{pq}]).
$$
\end{proposition}

\begin{proof} 
By Remark \ref{GysinFactorization}, 
we have a factorization $\jmath_{pq}=\xi _{pq\, !}\circ \zeta _{pq\, !}$. Again by Remark \ref{GysinFactorization}, $ \zeta _{pq\, !}$ is induced by the cup product with $c_{p-q}(S_{q}\slash S_{p})^c$.
As for the morphism $\xi _{pq\, !}$, we recall that it is just the Gysin morphism induced by the fibration
(\ref{projectionFpq}) 
$$\xi_{pq}:\fc_{pq}=\bF_{p-q, p-1}(\kc_p)\to \bG_{p-1}(\kc_p)= \tD_p,$$
hence it  is a  particular case of the one described in \cite[Exercise 3.8.3]{Manivel}.
Observe that $\xi_{pq}$ is a smooth fibration with fibres $\bG _{p-q}(\bC^{p-1})$,
so the relative  dimension is $(p-q)\cdot(q-1)$. 

In view of  (\ref{newcohex}), we have
$$ 
\jmath_{pq\, !}\circ \cup \theta_{ql}(\sigma_{\lambda}) \cup \theta_{pq}(\sigma_{\mu})=
\jmath_{pq\, !}\circ \cup s_{\lambda} \cup s_{\mu}, \quad \text{in} \quad Hom_{D^b_c(\Delta^0_{pl})}(R{\jmath_{pq} }{_*}\mathbb Q_{\tD_{pq}^l}\mid_{\tD_{pl}^0}, \mathbb Q_{\tD_{pl}^0}[2d_{pq}+\alpha +\beta] ),
$$
(here and in what follows we use the notations of the bivariant theory \cite[\S 2]{FultonCF}). 

Combining $\jmath_{pq}=\xi _{pq\, !}\circ \zeta _{pq\, !}$ with the notations of Remark \ref{ChernRootsFlag} and with \cite[p. 25]{FultonCF},
we get:
$$ 
\zeta _{pq}{_*}(\jmath_{pq\, !}\circ \cup s_{\lambda} \cup s_{\mu})=
\xi_{pq\, !}\circ \cup s_{\lambda} \cup s_{\mu}\cup s,
\quad \text{in} \quad Hom_{D^b_c(\Delta^0_{pl})}(R{\xi_{pq} }{_*}\mathbb Q_{\fc_{pq}}\mid_{\tD_{pl}^0}, \mathbb Q_{\tD_{pl}^0}[2d_{pq}+\alpha +\beta] ),
$$
where $s_{\mu}=s_{\mu}(t_{i_p+1}, \dots ,t_{i_q})$, $s_{\lambda}=s_{\lambda}(t_{i_q+1}, \dots ,t_{i_l})$
and $s:=c_{p-q}(S_q\slash S_p)^c$. By \cite[Remark 3.6.21]{Manivel}, the class $s_{\lambda}(t_{i_q+1}, \dots ,t_{i_l})$
can be extended to the class $s_{\lambda}(t_{i_q+1}, \dots ,t_{k})\in H^{\bullet}(\fc_{pq})$ in such a way the the restriction to $(\pi_{pl} \circ \xi_{pq})^{-1}(\Delta_l^0)$ is equivalent to
putting $t_{i_l+1}= \cdots = t_k=0$. Hence 
$
\xi_{pq\, !}\circ \cup s_{\lambda} \cup s_{\mu}\cup s
$
is well defined.

By \cite[Exercise 3.8.3, p. 148]{Manivel}, we conclude
$$
\xi _{pq}{_*}(\xi_{pq\, !}\circ \cup s_{\lambda} \cup s_{\mu}\cup s)=\cup s_{\nu}=\cup \theta_{pl}(\sigma_{\nu}),
\quad \text{in} \quad Hom_{D^b_c(\Delta^0_{pl})}(\mathbb Q_{\tD_{pl}^0}, \mathbb Q_{\tD_{pl}^0}[2d_{pq}+\alpha +\beta] ).
$$
We are done simply by applying the pull-back from $\Delta^0_l$.
\end{proof}

\medskip

\begin{theorem}
\label{LocStudy}
\begin{enumerate}
\item The complex $IC_{\Delta_p}^{\bullet}[-m_p]\mid _{\Delta_q^0}$ is quasi-isomorphic to a sub local system  of $R\pi_{pq}{_*} \bQ _{\tDpq ^0}$. More precisely, we have
$$IC_{\Delta_p}^{\bullet}[-m_p]\mid _{\Delta_q^0} \cong \sum_{\alpha=0}^{2\overline{k}_{pq}} B^{\alpha}_{pq}\otimes \mathbb Q_{\Delta_q^0}[-\alpha]
 \lhd 
\sum_{\alpha=0}^{2k_{pq}} A^{\alpha}_{pq}\otimes \mathbb Q_{\Delta_q^0}[-\alpha]
\cong R\pi_{pq}{_*} \bQ _{\tDpq ^0}.    $$
\item The canonical decomposition by supports of
$$A^{i+ m_p}\otimes_{\bQ }\bQ _{\Delta_q^0}\cong \sideset{_{}^{p}}{_{}^{i}}{\mathop \hc}(R\,\pi_{p *}\mathbb Q_{\tD_{p}}[m_{p}])\mid_{\Delta_q^0} \cong 
\bigoplus_{l\geq q}\left(\sideset{_{}^{p}}{_{}^{i}}{\mathop \hc}(R\,\pi_{p *}\mathbb Q_{\tD_{p}}[m_{p}])_{\Delta_l}\right)\mid_{\Delta_q^0}
$$
(cf. \cite[\S 1.1]{DeC}) coincides with the one given in (\ref{gammaAhigh})  and  is provided by the  morphisms $\gamma^q_{pr}$ ($i\geq 0$): 
\begin{equation}
\label{gammaAfinal}
A^{i+m_p}_{pq}\otimes \bQ _{\Delta_q^0}  \cong \bigoplus_{r=q+1}^{p-1}\bigoplus_{ \beta + \gamma =i+m_p} \left( \gamma^q_{pr}(D_{pr }^{ \beta-2d_{pr}}\otimes B_{rq}^{\gamma}\otimes \bQ _{\Delta_q^0}) \right) \oplus \gamma^0_{pq}
(D_{pq }^{i+m_p-2d_{pq}}\otimes \bQ _{\Delta_q^0}).
\end{equation}
Furthermore, the splitting (\ref{gammaAfinal}) is canonical namely independent of the cohomology extension.
\end{enumerate}
\end{theorem}

\begin{proof}
The proof is essentially identical  to the one of Theorem \ref{PerverseCohomology}, so we are going to be rather  sketchy.
The main point consists in observing that Proposition \ref{gamma} implies that the maps $\tilde{\gamma}^q_{pr}$ introduced on purely algebraic grounds in the proof of Therem
\ref{PerverseCohomology}  coincide with the ones defined before Proposition \ref{gamma}. Specifically, Proposition \ref{gamma} shows that 
$$
\gamma^q_{pr}(D_{pr }^{ \beta-2d_{pr}}\otimes B_{rq}^{\gamma}\otimes \bQ _{\Delta_q^0}) = \tilde{\gamma}^q_{pr}(D_{pr }^{ \beta-2d_{pr}}\otimes B_{rq}^{\gamma})\otimes \bQ _{\Delta_q^0}
\lhd A^{i+ m_p}\otimes_{\bQ }\bQ _{\Delta_q^0},
$$
as trivial local systems on $\Delta_q^0$. Thus (2) follows at once from (\ref{gammaAhigh}) and (2) follows just combining (\ref{gammaAlow}) with (1), similarly as in the last part of the proof of Theorem \ref{PerverseCohomology}.

Finally, the independence of the cohomology extensions follows exactly as in the proof of Theorem \ref{isoDE}.
\end{proof}

\medskip

\begin{remark}\label{recognize1}
Observe that Theorem \ref{LocStudy} is in accordance with the description of $IC_{\Delta_p}^{\bullet}[-m_p]\mid _{\Delta_q^0}$ obtained by means of the small resolution, thus answering the 
question raised in the introduction.
\end{remark}

\end{document}